\documentclass[12pt, letterpaper]{amsart}
\usepackage[margin=1in]{geometry}
\usepackage{amsmath, amssymb, amsfonts, mathtools}
\usepackage{dsfont}
\usepackage{enumitem}
\usepackage{caption}
\usepackage{subcaption}
\usepackage{float}
\usepackage{xcolor}

\newcommand{\dfn}[1]{\textcolor{red}{\emph{#1}}}
\definecolor{LightBlue}{rgb}{0.392,0.392,1} 
\usepackage[colorlinks=true, linkcolor=LightBlue, citecolor=LightBlue, urlcolor=LightBlue]{hyperref}
\usepackage[noabbrev,capitalise,nameinlink]{cleveref}
\renewcommand\thesubfigure{\Alph{subfigure}}
\crefformat{subfigure}{#2\figurename~\thefigure\thesubfigure#3}
\Crefname{subfigure}{Figure}{Figures}

\newtheorem{theorem}{Theorem}[section]
\newtheorem{lemma}[theorem]{Lemma}

\newtheorem{prop}[theorem]{Proposition}
\newtheorem{conj}[theorem]{Conjecture}

\newtheorem{definition}[theorem]{Definition}

\crefname{theorem}{Theorem}{Theorems}
\crefname{lemma}{Lemma}{Lemmas}
\crefname{claim}{Claim}{Claims}
\crefname{prop}{Proposition}{Propositions}
\crefname{conj}{Conjecture}{Conjectures}
\crefname{corollary}{Corollary}{Corollaries}
\crefname{remark}{Remark}{Remarks}
\crefname{definition}{Definition}{Definitions}
\theoremstyle{remark}
\newtheorem{remark}[theorem]{Remark}

\DeclareMathOperator{\mat}{Mat}

\DeclareMathOperator{\End}{End}
\DeclareMathOperator{\Tr}{Tr}
\DeclareMathOperator{\op}{op}

\DeclareMathOperator{\Ad}{Ad}
\DeclareMathOperator{\Prob}{\mathbb{P}}
\DeclareMathOperator{\GL}{GL}

\newcommand{\Hc}{\mathcal{H}}
\newcommand{\eps}{\varepsilon}
\newcommand{\EE}{\mathbb{E}}
\newcommand{\QQ}{\mathbb{Q}}
\newcommand{\CC}{\mathbb{C}}
\newcommand{\ZZ}{\mathbb{Z}}
\newcommand{\RR}{\mathbb{R}}
\newcommand{\TT}{\mathbb{T}}
\newcommand{\Ac}{\mathcal{A}}
\newcommand{\Fc}{\mathcal{F}}
\newcommand{\Gc}{\mathcal{G}}
\newcommand{\surjto}{\twoheadrightarrow}
\newcommand{\injto}{\hookrightarrow}
\newcommand{\1}{\mathds{1}}

\title{Taming Irrationality: An Invariance Principle for the Random Billiard Walk}
\author[Ruben Carpenter]{Ruben Carpenter}
\address[]{Department of Mathematics, Yale University, New Haven, CT 06511, USA}
\email{ruben.carpenter@yale.edu} 

\begin{document}

\begin{abstract}
    The \emph{random billiard walk} is a stochastic process $(L_t)_{t\geq 0}$ in which a laser moves through the Coxeter arrangement of an affine Weyl group in $\RR^d$, reflecting at each hyperplane with probability $p\in (0, 1)$ and transmitting otherwise. Defant, Jiradilok, and Mossel introduced this model from the perspective of algebraic combinatorics and established that, when the initial direction is rational, $L_t/\sqrt{t}$ converges in distribution to a spherical Gaussian. We bring analytic tools from ergodic theory and probability to the problem, and prove this central limit theorem for all initial directions. More strongly, we show the rescaled trajectories $t\mapsto n^{-1/2}L_{tn}$ converge to isotropic Brownian motion. Away from directions with rational dependencies, the covariance is continuous in $p$ and the initial direction.  
\end{abstract} 

\maketitle

\section{Introduction}

\subsection{The Random Billiard Walk}
\cref{fig: rational} depicts a laser beam moving through an infinite grid of equilateral triangles cut out by mirrors. It travels at unit speed. A probability $p\in (0, 1)$ is fixed, and every time the laser encounters the mirror, it either
\begin{itemize}
    \item reflects (angle of incidence equals angle of reflection) with probability $p$, or
    \item transmits straight with probability $1-p$.
\end{itemize}
Outcomes of distinct encounters are independent, and we assume no two mirrors are struck simultaneously. What can we say about the limiting distribution of the laser?

\begin{figure}[H]
    \centering
    \begin{subfigure}[b]{0.45\textwidth}
        \centering
        \includegraphics[width=\linewidth]{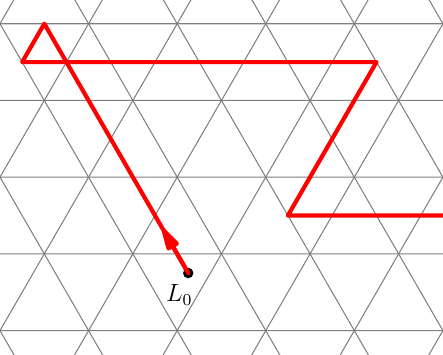}
        \caption{The random billiard walk in $\widetilde{W} = \widetilde{A}_2$.}
        \label{fig: rational}
    \end{subfigure}
    \hfill
    \begin{subfigure}[b]{0.45\textwidth}
        \centering
        \includegraphics[width=\linewidth]{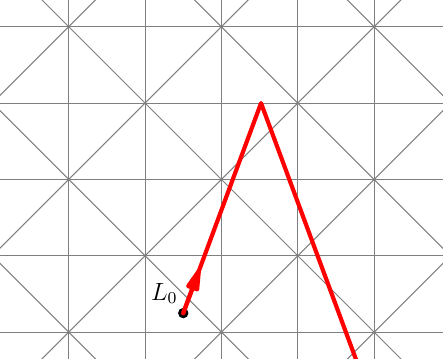}
        \caption{The random billiard walk in $\widetilde{W} = \widetilde{B}_2$.}
        \label{fig: irrational}
    \end{subfigure}
    \caption{Sample trajectories of the random billiard walk.}
    \label{fig: sample}
\end{figure}

This process can be defined for any hyperplane arrangement in Euclidean space. Let $\widetilde{W}$ be an affine Weyl group, whose Coxeter (hyperplane) arrangement acts on $\RR^d$. Defant, Jiradilok, and Mossel \cite{DJM25} introduced the \dfn{random billiard walk} $(L_t)_{t\geq 0}$, tracking the position of a laser following the reflection/transmission dynamics above. It is specified by an intial position $L_0\in \RR^d$, an initial direction $b$ in the unit sphere $S^{d-1}$, and $p\in (0, 1)$. Here we establish an invariance principle for this process, under any choice of initial parameters.

\subsection{Background and Context} 

This problem lies at the intersection of several active research areas. In dynamical systems, billiards constitute a classical and widely studied family of models, where researchers typically analyze the long-term behavior of a particle deterministically reflecting off domain boundaries using tools from ergodic theory and analysis.

Combinatorial billiards is a budding subfield of dynamical algebraic combinatorics, which studies discrete versions of these systems using algebraic and combinatorial methods. These often exhibit greater rigidity than their continuous counterparts, enabling more precise analysis. For example, \cite{DJ22, Z25} both establish sharp bounds on the number of cycles in a triangular grid billiards system. Barkley, Defant, Hodges, Kravitz, and Lee \cite{BDHKL25} classify the dynamics of \dfn{Bender–Knuth billiards}, where a discretized light beam moves through an arrangement of transparent and one-way mirrors. There has also been interest in randomized versions of combinatorial billiard systems: Defant introduced a \dfn{reduced random billiard walk} \cite{D24} as a variant of Lam's \dfn{reduced random walk} \cite{L15}.

The random billiard walk was introduced from this combinatorial perspective. Although the random billiard walk evolves in continuous time, the sequence of reflections applied encodes a discrete sequence of elements of the affine Weyl group. The analysis in \cite{DJM25}, however, relied on structural properties of the initial direction $b$, leaving much of the general picture open. In this work, we remove this restriction by combining machinery from algebraic combinatorics with analytic methods.

\subsection{Main Results}

We assume without loss of generality that the affine Weyl group $\widetilde{W}$ is irreducible. If this is not the case, then the system decomposes into independent components.

A direction $b$ is \dfn{rational} if it is rational with respect to the hyperplane arrangement, i.e. a scalar multiple of a vector in the coroot lattice $Q^\vee$. We say $b$ is \dfn{fully irrational} if its coordinates in \( Q^\vee \) are linearly independent over $\QQ$, and denote the set of such directions by $\mathcal{I}\subset S^{d-1}$. The direction in \cref{fig: rational} is rational; those in \cref{fig: irrational} or \cref{fig: RBW simulation} are not. 

For rational $b$, \cite{DJM25} proved that the rescaled position $L_t/\sqrt{t}$ converges in distribution to a multivariate Gaussian $\mathcal{N}(0, \sigma^2 I_d)$. Our first result extends this to all directions.

\begin{theorem}\label{thm: main theorem}
    Let \( L_0 \in \RR^d \), \( b \in S^{d-1} \), and \( p \in (0, 1) \). There exists $\sigma =\sigma_{b, p, L_0}> 0$ such that 
    \[
    \frac{L_t}{\sqrt{t}} \xrightarrow[]{\mathcal{D}} \mathcal{N}(0, \sigma^2 I_d).
    \]
    Moreover, for $(b, p)\in \mathcal{I}\times (0, 1)$ the variance $\sigma$ is a continuous function independent of $L_0$.
\end{theorem}

Our second result strengthens this to a functional central limit theorem. For each $n$, define 
the rescaled trajectory
\begin{equation}\label{eq: rescaled trajectory}
    L^{(n)}(t) := \frac{1}{\sqrt{n}} L_{nt} \qquad t\in [0, 1],
\end{equation}
which is a random variable in $C([0, 1], \RR^d)$, the Banach space of continuous functions from $[0, 1]$ to $\RR^d$ under the uniform topology.

\begin{theorem}\label{thm: main theorem 2}
     The rescaled trajectories $L^{(n)}\in C([0, 1], \RR^d)$ converge in distribution to a Brownian motion in $\RR^d$ with covariance $\sigma^2 I_d$. 
\end{theorem}

\begin{figure}[H]
    \centering
    \begin{subfigure}[b]{0.5\textwidth}
        \centering
        \includegraphics[width = \linewidth]{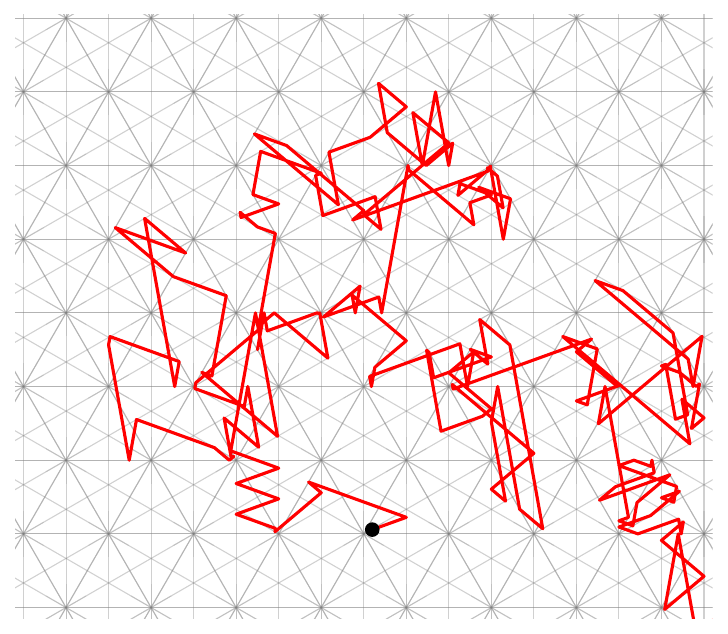}
        \caption{Simulated RBW in $\widetilde{G}_2$, $p = 0.3$, $200$ steps.}
        \label{fig: RBW simulation}
    \end{subfigure}
    \hfill
    \begin{subfigure}[b]{0.47\textwidth}
        \centering
        \includegraphics[width = \linewidth]{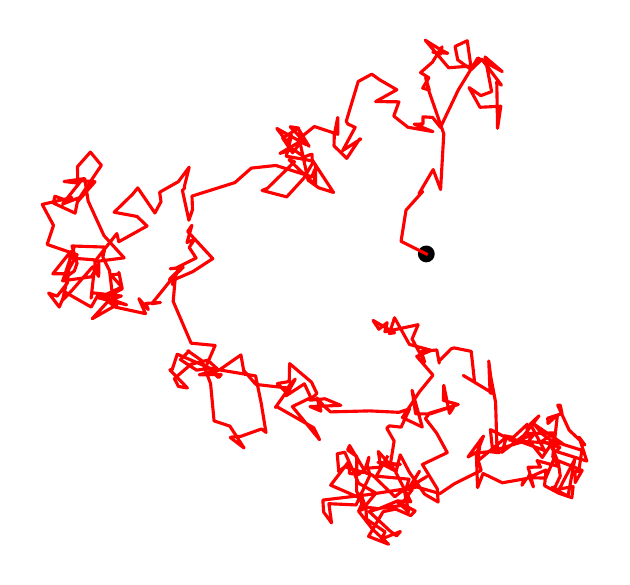}
        \caption{Simulated Brownian motion in $\RR^2$.}
        \label{fig: BM simulation}
    \end{subfigure}
    \caption{Illustration of the convergence in \cref{thm: main theorem 2}.}
    \label{fig:enter-label}
\end{figure}

\begin{remark}
    A classical theorem of Donsker \cite{P84} establishes this kind of convergence to a Brownian motion for the simple symmetric random walk on $\ZZ$. Note this can be viewed as a random billiard walk on the Coxeter arrangement of $\widetilde{A}_1$, where the mirrors are at integer points on $\RR$, and $p = 1/2$. Our theorem generalizes this classical invariance principle.
    
    Recall that the analysis of this walk naturally leverages the group structure of $\ZZ$ by breaking the walk into a sum of i.i.d.\ steps. Such a decomposition is unavailable here, but the underlying principle of using discrete algebraic structure remains central.
\end{remark}

\subsection{Roadmap} \label{sec: roadmap}

Let $\widetilde{W}$ be an irreducible affine Weyl group. We direct readers unfamiliar with these groups to \cref{sec: affine weyl groups}. While our results hold for all types, $\widetilde{A}_2$ (see \cref{fig: rational}) captures the relevant structure, and serves as a useful example to have in mind.

Each alcove has $d+1$ faces. A crucial algebraic observation from \cite{DJM25} is that there exists a $d+1$-coloring of the faces of the alcoves that is invariant under all reflections. We illustrate this in \cref{fig: numbering} and give a formal treatment in \cref{sec: background}. As a consequence, for any fixed $b$ and $L_0$, the cutting sequence $I = (i_1, i_2, i_3, \dots)$ of face labels is deterministic—it depends only on the direction $b$ and remains unchanged by reflections. This sequence is significant because the walk can be discretized by recording the center $X_n$ of the alcove the laser lies in after intersecting $n$ hyperplanes. Then $I$ becomes a ``sequence of instructions'': at time $n$, the walk crosses the face labeled with $i_n$ with probability $1-p$ (when it transmits), and stays in the same alcove with probability $p$ (when it reflects). In \cref{sec: discrete walk} we rigorously justify this, and argue that it suffices to understand the asymptotics of $(X_n)_{n\geq 0}$.

The direction $b$ is rational if and only if the cutting sequence $I$ is periodic. In that case, the displacement over each period is identically distributed (when conditioned on the direction at the start) and one can invoke standard central limit theorems as in \cite{DJM25}. For irrational directions, we no longer have these tools at our disposal. Instead, our goal will be to understand the asymptotic behavior of the walk through its global statistical features.

The convergence of the random billiard walk to a Brownian motion rests on two features: an unbiased direction and displacements that are asymptotically i.i.d. We first prove results establishing these. In \cref{sec: mixing} we show that the direction equidistributes exponentially fast, implying asymptotic independence of distant segments. In \cref{sec: walk grows} we lower-bound the growth rate of $X_n$, and in \cref{sec: ergodic} we show that the sequence $I$ is statistically regular. Together, these yield consistent covariance growth at the right scale in \cref{sec: covariance limit}.

In \cref{sec: higher moments} we show that all higher moments of $(X_n-X_0)/\sqrt{n}$ converge to those of a Gaussian, giving \cref{thm: main theorem}. To prove \cref{thm: main theorem 2} we need to consider the whole trajectory at once. We do this in \cref{sec: martingale approximation} by constructing a martingale that closely tracks $(X_n)_{n\geq 0}$, and invoking standard invariance principles. We retain \cref{sec: higher moments} for its independent interest and potential applicability to other models we introduce in \cref{sec: further directions}.

\section{Preliminaries and Background} \label{sec: background}

\subsection{Affine Weyl Groups}\label{sec: affine weyl groups}

A standard reference is~\cite{BB05}. These groups are generated by reflections across a hyperplane arrangement in $\RR^d$ positioned according to a root system.

\begin{definition}
    A finite set $\Phi \subset \RR^d$ spanning $\RR^d$ is a \dfn{root system} if:
    \begin{enumerate}
        \item \label{item: refl} For every $\alpha \in \Phi$, the set $\Phi$ is preserved under reflection about the hyperplane 
        \[
        H_\alpha^0 := \{x \in \RR^d : x^\top \alpha = 0 \}
        \]
        \item For all $\alpha, \beta \in \Phi$, the projection of $\beta$ on $\alpha$ is a half-integer multiple of $\alpha$. 
        \label{item: proj}
        \item If $\alpha \in \Phi$, the only other  multiple of $\alpha$ in $\Phi$ is $-\alpha$. 
    \end{enumerate}
\end{definition}

The \dfn{Weyl group} $W$ is generated by the reflections across all $H_\alpha^0$, which is finite by~\eqref{item: refl}. The action of $W$ on $\RR^d$ gives a representation $\rho\colon W \injto \GL(d, \RR)$. If $\rho$ is irreducible, $W$ is \dfn{irreducible}; such groups have been classified~\cite{BB05}. 

The \dfn{affine Weyl group} $\widetilde{W}$ is generated by reflections across all affine hyperplanes
\[
H_\alpha^k := \{x \in \RR^d : x^\top \alpha = k \}, \quad \alpha \in \Phi, \; k \in \ZZ.
\]
The set $\Hc$ of all $H_\alpha^k$ is called the \dfn{Coxeter arrangement} of $\widetilde{W}$. By~\eqref{item: proj} it partitions $\RR^d$ into congruent simplices called \dfn{alcoves}. The action of $\widetilde{W}$ on $\RR^d$ freely and transitively permutes the alcoves. Fixing a \dfn{fundamental alcove} $\Ac_0$, the map $w \mapsto w\cdot \Ac_0$ gives a bijection between elements of the affine Weyl group and alcoves. 

Let $s_0, \dots, s_d$ be the reflections about faces of $\Ac_0$. They generate $\widetilde{W}$. Moreover, the alcove $w \Ac_0$ is adjacent to the alcoves $ws_i \Ac_0$ for $0\leq i\leq d$. We now construct the labeling we need.

\begin{lemma}\label{lmm: labeling}
    Label the face between $w \Ac_0$ and $ws_i \Ac_0$ with $i$. This labeling is $\widetilde{W}$-invariant.
\end{lemma}

\begin{proof}
    Consider any $u \in \widetilde{W}$. Then the face between alcoves $uw\cdot \Ac_0$ and $uws_i\cdot\Ac_0$ is also labeled with $i$. Thus $u$ sends faces labeled with $i$ to faces labeled with $i$, as required. 
\end{proof}

For $\alpha \in \Phi$, its \dfn{coroot} is $\alpha^\vee := \frac{2}{\langle \alpha, \alpha \rangle} \alpha$. Note $H_\alpha^k = H_\alpha^0 + k\alpha^\vee/2$. Hence, reflecting about $H_\alpha^0$ and then $H_\alpha^1$ shows translation by $\alpha$ is in an element of the affine Weyl group. In fact
$$Q^\vee := \mathrm{span}_\ZZ \{\alpha^\vee : \alpha \in \Phi\}\subset \widetilde{W}$$
is a lattice called the \dfn{coroot lattice}, and $\widetilde{W} = W \ltimes Q^\vee$. Geometrically this algebraic fact means the orbit of $\Ac_0$ under $W$ forms a polytope whose $Q^\vee$-translates tile $\RR^d$ (see \cref{fig: DXn picture}).

\subsection{Probability Theory} \label{sec: probability}

A \dfn{probability space} consists of a sample space \(\Omega\), a $\sigma$-algebra \(\mathcal{F}\), and a measure \(\mathbb{P}\). Here $\Omega = (\Omega, \delta)$ will always be a Polish space; either $\RR$, $\RR^d$ or $C([0, 1], \RR^d)$. 

\begin{definition}
    A random variable $X$ in $\RR^d$ follows a \dfn{normal distribution} \(\mathcal{N}(\mu, \Sigma)\) with mean \(\mu \in \RR^d\) and positive definite covariance \(\Sigma \in \mat_{d\times d}(\RR)\) if
    \[
    \mathbb{P}(X \in A) = \int_A \frac{1}{(2\pi)^{d/2} \sqrt{\det \Sigma}} \exp\left(-\frac{1}{2} (x - \mu)^\top \Sigma^{-1} (x - \mu)\right) \, \mathrm{d}x
    \]
    for every Borel set \(A \subseteq \RR^d\).
\end{definition}

A Brownian motion is a sequence of normal random variables in $C([0, 1], \RR^d)$.

\begin{definition} $(B_t)_{t\in[0, 1]}$ in $\RR^d$ is a \dfn {Brownian motion} with covariance $\Sigma\in \mat_{d\times d}(\RR)$ if:
\begin{enumerate}
    \item For all $0 \leq s < t\leq 1$, $B_t - B_s\sim \mathcal{N}(0, (t-s)\Sigma)$. 
    \item For all $0 =  t_0 < t_1 < \dots < t_k\leq 1$, the increments $$B_{t_1} - B_{t_0}, \quad B_{t_2} - B_{t_1}, \quad \dots, \quad B_{t_k} - B_{t_{k-1}}$$ 
    are independent.
    \item The sample path $t \mapsto B_t$ is almost surely continuous.
\end{enumerate}
\end{definition}

A sequence $(X_n)_{n\geq 0}$ \dfn{converges in probability} to \(X\), written $X_n \xrightarrow{P} X$, if for every $\varepsilon>0$
$$\lim_{n\to\infty} \mathbb{P}\big(\delta(X_n,X)>\varepsilon\big) = 0.$$
It \dfn{converges in distribution} to $X$, written $X_n \xrightarrow[]{\mathcal{D}} X$, if for every continuous \(\varphi:S\to\mathbb{R}\)
$$\lim_{n\to\infty} \mathbb{E}[\varphi(X_n)] = \mathbb{E}[\varphi(X)].$$ 
This is weaker than convergence in probability. We repeatedly use the following result.

\begin{lemma}[Slutsky's lemma] \label{lmm: Slutsky}
    Let $(X_n)_{n\geq 0}$ and $(Y_n)_{n\geq 0}$ be random variables on a Polish space with metric $d$. Suppose $X_n\xrightarrow[]{\mathcal{D}} X$ and $\delta(X_n, Y_n) \xrightarrow[]{P} 0$. Then $X_n+Y_n \xrightarrow[]{\mathcal{D}} X$. 
\end{lemma}

A \dfn{filtration} \((\mathcal{F}_t)_{t \geq 0}\) is a family of $\sigma$-algebras such that \(\mathcal{F}_s \subseteq \mathcal{F}_t\) for \(s \leq t\). A process \((X_t)_{t \geq 0}\) is \dfn{adapted} to \((\mathcal{F}_t)_{t \geq 0}\) if each \(X_t\) is \(\mathcal{F}_t\)-measurable. It is a \dfn{martingale} if $\mathbb{E}[X_t \mid \mathcal{F}_s] = X_s$ for all $s\leq t$, and \(\mathbb{E}[\delta(X_t, 0)] < \infty\). Note Brownian motion is a martingale.

\begin{remark} \label{rmk: convexity}
    Throughout, we establish bounds of the form $\|\EE[X\mid \Fc]\|\leq C$. By the tower property and Jensen's inequality, this automatically holds for any sub-$\sigma$-algebra $\Gc\subseteq \Fc$:
    $$\|\EE[X\mid \Gc]\| = \|\EE[\EE[X\mid \Fc]\mid \Gc]\| \leq \EE[\|\EE[X\mid \Fc]\|\mid \Gc]\leq C.$$
\end{remark}

\subsection{Discretizing the Random Billiard Walk} \label{sec: discrete walk}

Let $X_n\in\RR^d$ record the \emph{centroid} of the alcove the laser lies in after the $n$-th intersection, as shown in \cref{fig: numbering}. We assume $X_0\in \Ac_0$. Then, by previous discussions, $X_n$ lies in the random alcove $w_n\Ac_0 := s_{i_1}^{\eps_1}\cdots s_{i_{n}}^{\eps_{n}}\Ac_0$, where
$$\eps_n := \begin{cases}
    0 & \text{the walk reflects the $n$-th time it hits a hyperplane}\\
    1 & \text{otherwise}
\end{cases} \sim \text{Ber}(1-p)
$$
are i.i.d.\ Bernoulli random variables. Let $\Fc_n := \sigma(X_0, \dots, X_n)$ be the natural filtration.

\begin{figure}[H]
    \centering
    \includegraphics[width=0.9\linewidth]{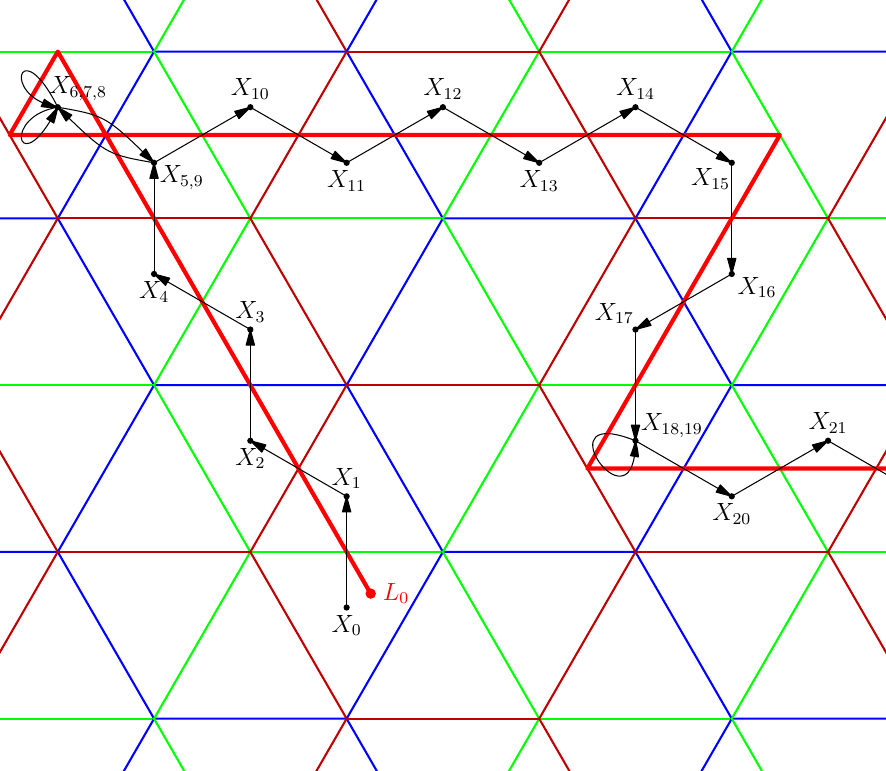}
    \caption{The coloring of the alcove faces is invariant under all reflections in $\widetilde{A}_2$. This rational direction has cutting sequence period: \textcolor{green}{green}, \textcolor{red}{red}, \textcolor{blue}{blue}.}
    \label{fig: numbering}
\end{figure}

We now argue that it is enough to study $(X_n)_{n\geq 0}$. As in \eqref{eq: rescaled trajectory} we define a continuous rescaled trajectory $X^{(n)}$, by first linearly interpolating. This can be directly compared to $L^{(n)}$. 

\begin{prop} \label{lmm: X to L}
    If $X^{(n)}$ converges to nondegenerate Brownian motion, then so does $L^{(n)}$. 
\end{prop}
\begin{proof}
    Since the $\Hc$ is invariant under reflections, the number of hyperplanes hit in the first $t$ seconds is independent of reflections. If the laser had never reflected, it would at $L_t = L_0+tb$, intersecting
    $t|b^\top \alpha^\vee | + O(1)$ hyperplanes perpendicular to $\alpha$. Thus, if we define
    $$k_b:= \frac{1}{4}\sum_{\alpha \in \Phi} |b^\top \alpha^\vee| >0,$$
    after $t$ seconds the laser has hit $k_bt + O(1)$ hyperplanes. In particular, $L_t = X_{\lfloor k_bt\rfloor} + O(1)$. 
    Therefore, we can bound the difference between $L^{(n)}$ and the $n$-th linear interpolation as
    $$\left\|L^{(n)} - \sqrt{k_b}X^{(k_bn)}\right\|_\infty = \sup_{t\in [0, 1]}\left\|\frac{1}{\sqrt{n}}L_{tn} - \frac{1}{\sqrt{n}}X_{tk_bn}\right\|= O(n^{-1/2}) \to 0.$$
    The conclusion follows from \cref{lmm: Slutsky}.
\end{proof}

\section{Direction Mixes Exponentially Quickly} \label{sec: mixing}

For now, we forget about the affine displacement of $(X_n)_{n\geq 0}$, and zoom in on its direction. Define $\Delta X_n := X_{n+1} - X_n$, the step taken at time $n$. At time $n$ the laser hits a face labeled with $i_n$. Since the labeling is $Q^\vee$-invariant, the direction of $\Delta X_n$ will only depend on the image of $w_{n-1}\Ac_0$ under the projection $\pi: \widetilde{W}\surjto \widetilde{W}/Q^\vee = W$. This is illustrated in \cref{fig: DXn picture}. 

\begin{figure}[H]
    \centering
    \includegraphics[width=0.9\linewidth]{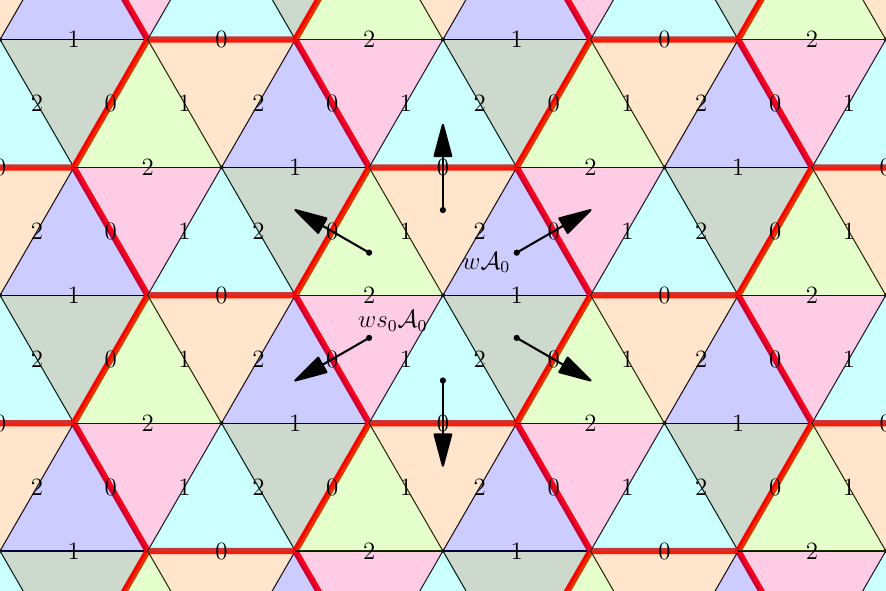}
    \caption{In $W = \widetilde{A}_2$, if we know $i_n = 0$, there are 6 possibilities for $\Delta X_n$, that depends on the class of $w_{n-1}$ modulo $Q^\vee$.}
    \label{fig: DXn picture}
\end{figure}

For each $0\leq i \leq d$ let $\beta_i$ denote the vector connecting the center of $\Ac_0$ to that of $s_i \Ac_0$ (in particular, it is parallel to $\alpha_i$). The discussion above shows
\begin{equation}\label{prop: DXn}
    \Delta X_n = \eps_n \rho( \pi(w_{n-1})) \beta_{i_n}.
\end{equation}
Note $\pi(w_n)= \pi(s_{i_1})^{\eps_1}\cdots \pi(s_{i_n})^{\eps_n}$ is a Markov chain on $W$. We claim it mixes quickly. 

\begin{prop}\label{lmm: quick mixing}
    There exists $c\in (0, 1)$ such that, for all $n, n_0 \geq 0$ and $w\in W$ we have  $$\left|\Prob\left[\pi(w_{n+n_0}) = w \mid \Fc_{n_0}\right] - \frac{1}{|W|}\right| \leq c^{n}.$$
    Moreover, $c = c_p$ is independent of $b$ and $L_0$, and is continuous in $p$. 
\end{prop}

Since $I$ determines the process, it is inherently time-dependent. Thus, while our analysis borrows standard ideas for random walks on finite groups, it requires ad-hoc arguments.

\subsection{Proof of \texorpdfstring{\cref{lmm: quick mixing}}{quick mixing}}

The evolution of $(\pi(w_n))_{n\geq 0}$ is most naturally written in the group ring $\RR[W]$, representing the distribution of a random variable $X$ on $W$ as
$$P = \sum_{w\in W} \Prob[X = w]\cdot w.$$
Conditioning gives a distribution $P_{n_0}\in \RR[W]$ at time $n_0$. We have
the simple recurrence
$$P_{m} = P_{m-1}\cdot (p + (1-p)\pi(s_{i_{m}})).$$
Let $\1\in \RR[W]$ represent the uniform distribution on $W$, and $\|\cdot\|_W$ the $\ell_2$ norm on $\RR[W]$ with respect to the standard basis. We prove $\|P_{n+n_0}-\1\|_W\to 0$ exponentially quickly in $n$. For $0\leq i \leq d$, right multiplication by $p + (1-p) \pi(s_i)$ defines a linear map $T_i$ on $\RR[W]$. We show each is \textit{weakly} a contraction towards $\1$, and that together they contract everything. 

\begin{lemma} \label{lmm: full contraction}
    For each \( i \), there is an orthogonal decomposition $\RR[W] = F_i \oplus C_i$ such that \( T_i \) acts as the identity on \( F_i \), and contracts every element of \( C_i \) by a factor of \( |2p - 1| < 1 \).

    Moreover, the common fixed space $\bigcap_{i = 0}^r F_i$ is spanned by the uniform distribution.
\end{lemma}
\begin{proof}
    Since \( s_i^2 = 1 \), right multiplication by \( \pi(s_i) \) defines an involutive linear operator on \( \RR[W] \). Its matrix with respect to the standard basis \( \{ w : w \in W \} \) is symmetric: it permutes basis vectors in pairs, sending \( w \mapsto w\pi(s_i) \) and vice versa. Moreover, its eigenvalues lie in \( \{ \pm 1 \} \), so by the spectral theorem we have an orthogonal decomposition $\RR[W] = F_i \oplus C_i$
    where \( F_i \) and \( C_i \) are the $\pm 1$ eigenspaces of right multiplication by \( \pi(s_i) \), respectively. Then $T_i$ acts by 1 on $F_i$ and $p-(1-p) = (2p-1)$ on $C_i$. For the second part, suppose 
    $$P = \sum_{w \in W} p_w \cdot w$$ 
    is fixed by every \( T_i \). It follows that $p_{w \pi(s_i)} = p_w$ for all \( w \in W \) and $0\leq i \leq d$. Since \( W \) is generated by the \( \pi(s_i) \), this implies that $p_w$ is constant on $W$. Thus, $P \in \RR \cdot \1$.
\end{proof}

Since $|2p-1| <1$ and $F_i\perp C_i$, for every $P \in \RR[W]$ we have $\|T_iP\|_W\leq \|P\|_W$ with equality if and only if $P\in F_i$. Therefore, after applying each $T_i$ we get strict contraction. 

\begin{lemma}\label{prop: is contraction}
    Suppose the sequence $(\iota_1, \dots, \iota_m)$ contains each integer in $\{0, 1, \dots, d\}$ at least once. Then there exists $c_{\iota_1, \dots, \iota_m} \in (0, 1)$ such that
    $$\left\|T_{\iota_m}\cdots T_{\iota_1}P \right\|_W\leq c_{\iota_1, \dots, \iota_m} \cdot \left\|P\right\|_W$$
    for every $P$ with $(P, \1) = 0$. 
\end{lemma}
\begin{proof}
    Consider the sphere $\mathcal{S} := \{P \in \RR[W]: \|P\|_W = 1, (P, \1) = 0\}$. The function
    \begin{align*}
        \mathcal{S} &\to \RR \\
        P &\mapsto \|T_{\iota_m}\cdots T_{\iota_1}P\|_W
    \end{align*}
    is continuous.  Since $\mathcal{S}$ is compact, its supremum $c_{\iota_1, \dots, \iota_m}$ is attained. Therefore, it suffices to prove the uniform bound $\|T_{\iota_m}\cdots T_{\iota_1}P\| < 1$ for every $P\in \mathcal{S}$. We know
    $$\|T_{\iota_m}\cdots T_{\iota_1}P\|_W\leq \|T_{\iota_{m-1}}\cdots T_{\iota_1}P\|_W \leq \cdots \leq \|P\|_W = 1$$
    with equality if and only if $T_{\iota_i}\cdots T_{\iota_1}P = P$ for all $i$, i.e. $T_{\iota_j}P = P$ for all $j$. In this case $P$ must be invariant under every $T_i$, so $P\in \RR \cdot \1$ by \cref{lmm: full contraction}. Then $P = 0$, contradiction.
\end{proof}

\begin{remark}
    In general, right multiplication by a distribution $P =\sum_{w\in W} p_x$ is always a weak contraction towards $\1$ in the $\ell_1$ norm. Moreover, if we let $\Sigma := \{w : p_w > 0\}$ be its support, then it is a strict contraction if and only if $\Sigma$ generates $W$, and is not contained in any nontrivial coset of a normal subgroup of $W$.

    Thus, an alternative proof of \cref{prop: is contraction} is to note that the support of the distribution
    $$P = (p\cdot 1+ (1-p)\cdot s_{\iota_1})\cdots (p\cdot 1+ (1-p)\cdot s_{\iota_m})$$
    generates $W$ (since by assumption on $\iota_1, \dots, \iota_m$ it contains $s_0, \dots, s_d$), and also satisfies the second condition since it contains 1 (so the only coset of a normal subgroup it could be contained in is $W$). We include the original proof since it is more elementary and insightful.
\end{remark}
    
We claim there is $M$ such that among any $M$ indices $i_{n+1}$, \dots, $i_{n+M}$ every integer in $\{0, 1, \dots, d\}$ appears. This is a geometric fact: for each $0\leq i \leq d$, the faces labeled with $i$ delimit congruent copies of a polytope $\mathcal{P}_i$ that tiles $\RR^d$. Each of them contains a finite number of faces inside; let $M$ be larger than this quantity for every $i$. 

Using the constants from~\cref{prop: is contraction} set
$$\eps:= \max_{(\iota_1, \dots, \iota_M)} c_{\iota_1, \dots, \iota_M} < 1,$$
where the maximum is taken over all sequences containing each integer in $\{0, 1, \dots, d\}$. By definition, every $M$ steps the distribution contracts to $\1$ by a factor of at least $\eps$. Thus
$$\|P_{n+n_0} - \1\|_{\infty}\leq \|P_{n+n_0} - \1\|_W \leq \eps^{\left\lfloor \frac{n}{M}\right\rfloor}\|P_{n_0} - \1\|_W,$$
which gives exponential decay. Moreover, for every initial distribution $P_{n_0}$ we have
$$\|P_{n_0} - \1\|_W \leq \sqrt{1-1/|W|} < 1.$$ 
Also $\eps^{\left\lfloor \frac{n}{M}\right\rfloor}\leq \eps^{-M} (\eps^{1/M})^{n}$. Thus, there is some $c < 1$ for which the upper bound is $c^n$. Since $c_{\iota_1, \dots, \iota_M}$ depends continuously on $p$ and is independent of $b$, $L_0$, the same is true for $c$. 

\subsection{Geometric Consequences of Quick Mixing}

\cref{lmm: quick mixing} shows the walk quickly forgets previous information. We now use the symmetry of $\Hc$ to argue that, in expectation, each step doesn't move much. We use the norm $\|x\| = \sqrt{x^\top x}$.

\begin{lemma} \label{prop: independence}
    For any $n, n_0\geq 0$ we have
    $$\|\EE[\Delta X_{n+n_0}\mid\Fc_{n_0}] \| \leq (1-p)|W|c^{n}.$$
\end{lemma}
\begin{proof}
    The $n+n_0$-th step crosses the face labeled $i_{n+n_0}$ with probability $1-p$, and else stays still. As in \cref{fig: DXn picture}, for each $w$ and $s_i$ the displacements $w \Ac_0 \to ws_{i}\Ac$ and $ws_{i}\Ac \to \Ac$ are opposite. By \cref{lmm: quick mixing} the probabilities they happen differ by at most $2c^{d}$, so
    \begin{align*}
        &\|\EE[\Delta X_{n+n_0}\mid\Fc_{n_0}] \| \\&= (1-p)\left\|\sum_{w\in W} \Prob[\pi(w_{n+n_0-1}) = w\mid \Fc_{n_0}] \rho(w_{n+n_0-1}) \beta_{i_{n+n_0}}\right\| \\
        &\leq \frac{1-p}{2}\sum_{w\in W} \left|\Prob[\pi(w_{n+n_0-1}) = w\mid \Fc_{n_0}] -\Prob[\pi(w_{n+n_0-1}) = ws_{i_{n+n_0}}\mid \Fc_{n_0}]\right| \|\rho(w_{n+n_0-1}) \beta_{i_{n+n_0}}\| \\
        &\leq \frac{1-p}{2}\cdot 2|W|\cdot c^n,
    \end{align*}
    as claimed. 
\end{proof}

It follows that the walk is almost a martingale.

\begin{lemma} \label{prop: almost martingale}
    There is an absolute constant $C$ such that for any $n, n_0\geq 0$
    $$\|\EE[X_{n+n_0} \mid \mathcal{F}_{n_0}] - X_n\| \leq C$$
    Moreover, $C$ is a locally bounded function of $p$. 
\end{lemma}
\begin{proof}
    By linearity of expectation and \cref{prop: independence}
    \begin{align*}
        \left\|\EE[X_{n+n_0}\mid \mathcal{F}_{n_0}] - X_n\right\| &= \left\|\sum_{i=0}^{n-1}\EE\left[\Delta X_{i+n_0}\mid\mathcal{F}_{n_0}\right]\right\| \leq \sum_{i=0}^{n-1} (1-p) |W| c^{i},
    \end{align*}
    which is upper bounded by a convergent sum that is locally bounded in $p$. 
\end{proof}

These two structural lemmas are really important, and we will use them repeatedly.

\section{The Growth Rate of the Random Billiard Walk} \label{sec: walk grows} 

Ultimately, we have to argue that the correct scaling is $X_N/N^{1/2}$, i.e. that this is the rate at which the walk diverges from the origin. In \cref{sec: covariance limit} we show the scaled covariance has a finite limit $\sigma^2 I_d$, but our proof will not show $\sigma\neq 0$. Proving positivity requires using the geometry of the affine Weyl group in a nontrivial way: for instance, we need to rule out the case that the steps are $\Delta X_1$, $-\Delta X_1$, $\Delta X_2$, $-\Delta X_2$, \dots. 

We begin with a weaker result showing the walk ``gets off the ground.'' 

\begin{lemma}\label{lmm: goes to 0}
    For each alcove $\Ac$, $\Prob[X_{n+n_0}\in\Ac\mid \Fc_{n_0}] \to 0$ as $n\to \infty$, uniformly in $n_0$.
\end{lemma}

One might hope to reuse the idea from~\cref{sec: mixing} on the distribution of $w_n$ in the group ring $\RR[\widetilde{W}]$. This fails, because the affine Weyl group is infinite and the relevant convolution operators are not compact. Instead, we consider large, but finite, quotients of the space.

\begin{proof}[Proof of \cref{lmm: goes to 0}]
    Fix $\lambda\in \mathbb{N}$ and consider the quotient $\pi_\lambda : \widetilde{W}\to \widetilde{W}/\lambda Q^\vee$ (since $Q^\vee$ is normal, so is $\lambda Q^\vee$). We identify the codomain with a finite set $\mathcal{D}$ of alcoves, as in \cref{fig: fundamental domain}. 
    
    \begin{figure}[H]
        \centering
        \includegraphics[width=0.5\linewidth]{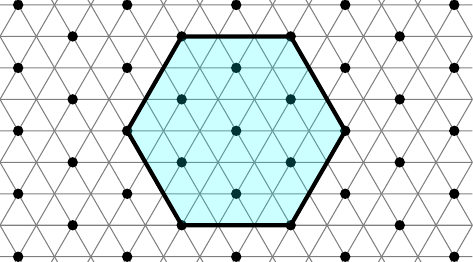}
        \caption{The set of alcoves $\mathcal{D}$ representing the quotient group $\widetilde{A}_2/3Q^\vee$.}
        \label{fig: fundamental domain}
    \end{figure}

    Then $\pi_\lambda(\Ac)\in \mathcal{D}$ and $$\Prob[X_{n+n_0} \in \Ac\mid \Fc_{n_0}] \leq \Prob[\pi_\lambda(X_{n+d}) \in \pi_\lambda(\Ac)\mid \Fc_{n_0}].$$ 
    Since $\pi_\lambda(s_0)$, $\pi_\lambda(s_1)$, \dots, $\pi_\lambda(s_d)$ generate $\pi_\lambda(\widetilde{W})$, the argument from \cref{sec: mixing} shows that $\pi_\lambda(X_{n+n_0})$ equidistributes on $\mathcal{D}$. Thus
    $$\Prob\left[\pi_\lambda(X_{n+n_0}) \in \pi_\lambda(\Ac)\mid \Fc_{n_0}\right] = \frac{1}{|\mathcal{D}|} + o(1).$$
    and the convergence rate is uniform in the initial distribution given by $\Fc_{n_0}$. Finally, take $\lambda\to \infty$, so that $|\mathcal{D}|\to \infty$. The conclusion follows. 
\end{proof}

Thus $(X_n)_{n\geq 0}$ spreads out in $\RR^d$. We make this quantitative in terms of $\EE[\|X_n-X_0\|^2]$, the moment of inertia of the probability distribution of $X_n$ about $X_0$. More strongly, for each $n\geq 0$, let $f(n)$ be the largest real number such that
$$\EE\left[\|X_{n+n_0}-X_{n_0}\|^2\mid \Fc_{n_0}\right] \geq f(n)$$
for all $n_0\geq 0$; in particular $\EE\left[\|X_n-X_0\|^2\right]\geq f(n)$. A corollary of \cref{lmm: goes to 0} is $f(n)\to \infty$. We bootstrap this in \cref{prop: induct} to prove the following result.

\begin{prop}\label{prop: fn grows}
    We have
    $$\liminf_{n\to \infty} f(n)/n > 0.$$
\end{prop}

To build intuition for the statement and proof of \cref{prop: induct}, consider the possibilities for $X_{m+n}$ for a fixed position of $X_m$. \cref{prop: almost martingale} shows $\EE[X_{m+n} -X_m\mid X_m]$ is small. But
$$\EE[\|X_{m+n}-X_m\|^2\mid X_m]\geq f(n).$$
Thus, if $f(n)$ is large, the possible sample paths from $X_m$ to $X_{m+n}$ must be long and point in different directions to balance out. This means the distribution of $X_{m+n}$ spreads out, giving a lower bound on $\EE[\|X_{m+n}-X_0\|^2]$ in terms of $\EE[\|X_m-X_0\|^2]$ and $f(n)$. 

\begin{lemma}\label{prop: induct}
    Let $C$ be the constant from~\cref{prop: almost martingale}. For any $n, m\geq 0$ we have
    $$f(n+m) \geq f(n) + f(m)-2C\sqrt{f(m)}$$
    provided $f(m)\geq C^2$. 
\end{lemma}
\begin{proof}
    Fix any $n_0 \geq 0$. We have
    \begin{align*}
        &\|X_{n+m+n_0}-X_{n_0}\|^2 \\ &= \|X_{n+m+n_0}-X_{m+n_0}\|^2 + \|X_{m+n_0}-X_{n_0}\|^2 + 2\langle X_{m+n_0}-X_{n_0}, X_{n+m+n_0}-X_{m+n_0}\rangle.
    \end{align*}
    Take the conditional expectation with respect to $\Fc_{n_0}$. By definition 
    $$\EE\left[ \|X_{n+m+n_0}-X_{m+n_0}\|^2\mid \Fc_{n_0}\right] \geq f(n).$$
    On the other hand, from the Cauchy--Schwarz inequality
    \begin{align*}
        &\EE\left[\langle X_{m+n_0}-X_{n_0}, X_{n+m+n_0}-X_{m+n_0}\rangle\mid  \Fc_{n_0}\right]  \\
        &= \EE\left[\langle X_{m+n_0}-X_{n_0}, \EE[X_{n+m+n_0}-X_{m+n_0} \mid X_{m+n_0}\vee \Fc_{n_0}] \rangle \mid \Fc_{n_0}\right] \\
        &\geq \EE\left[-\|X_{m+n_0}-X_{n_0}\|\cdot \|\EE[X_{n+m+n_0}-X_{m+n_0}\mid X_{m+n_0}\vee\Fc_{n_0}]\| \mid \Fc_{n_0}\right].
    \end{align*}
    Since $\sigma(X_{m+n_0}, \Fc_{n_0})\subseteq \Fc_{m+n_0}$, by \cref{prop: almost martingale} and \cref{rmk: convexity}
    $$\|\EE[X_{n+m+n_0}-X_{m+n_0}\mid X_{m+n_0}, \Fc_{n_0}]\| \leq C.$$
    Thus, using Cauchy--Schwarz again,
    \begin{align*}
        \EE\left[\langle X_{m+n_0}-X_{n_0}, X_{n+m+n_0}-X_{m+n_0}\rangle \mid \Fc_{n_0}\right] &\geq -C\cdot \EE\left[\|X_{m+n_0}-X_{n_0}\|\mid \Fc_{n_0}\right] \\
        &\geq -C\sqrt{\EE\left[\|X_{m+n_0}-X_{n_0}\|^2\mid \Fc_{n_0}\right]}.
    \end{align*}
    Therefore
    \begin{align*}\label{eq: bound}
        &\EE \left[\|X_{n+m+n_0}-X_{n_0}\|^2 \mid \Fc_{n_0}\right] \\
        &\geq f(n) + \EE\left[\|X_{m+n_0}-X_{n_0}\|^2\mid \Fc_{n_0}\right] - 2C\sqrt{\EE\left[\|X_{m+n_0}-X_{n_0}\|^2\mid \Fc_{n_0}\right]}.
    \end{align*}
    By definition $\EE\left[\|X_{m+n_0}-X_{n_0}\|^2\mid \Fc_{n_0}\right] \geq f(m)$. If $\sqrt{f(m)}\geq C$, then 
    $$\EE\left[\|X_{m+n_0}-X_{n_0}\|^2\mid \Fc_{n_0}\right] - 2C\sqrt{\EE\left[\|X_{m+n_0}-X_{n_0}\|^2\mid \Fc_{n_0}\right]} \geq f(m) - 2C\sqrt{f(m)},$$
    which yields the claim. 
\end{proof}

We can now lower bound the growth rate of the random billiard walk.

\begin{proof}[Proof of \cref{prop: fn grows}]
    Because $f(n)\to \infty$  we know $K:=f(m_0) - 2C\sqrt{f(m_0)} > 0$ for some $m_0$. Then for any $n = tm_0+r$ ($0\leq r < m_0$) repeatedly applying~\cref{prop: induct} yields
    $$f(n)\geq  t\left(f(m_0) - 2C\sqrt{f(m_0)}\right) +f(r)\geq tK\geq \frac{K}{m_0}n - K,$$ 
    so $f(n)$ is lower-bounded by a linear function, and so $\liminf_{n\to \infty} f(n)/n > K/m_0 > 0$.
\end{proof}

\section{Ergodic Properties of the Cutting Sequence \texorpdfstring{$I$}{I}} \label{sec: ergodic}

Since the cutting sequence $I=(i_1, i_2, \dots)$ of face labels specifies the discretized random billiard walk, it is important to understand its structure. We already noted that it is periodic if and only if the direction $b$ is rational; in general it will look locally irregular. A key ingredient in the proof of \cref{thm: main theorem} is that it is nonetheless statistically regular.

\begin{prop}\label{lmm: equidistribution of windows}
    Given a pattern $\iota \in \{0, 1, \dots, d\}^{d+1}$, there is $p_{\iota, b}\in [0, 1]$ such that
    \[
    \sup_{n_0 \geq 0} \left| \frac{1}{N} \# \left\{n_0 \leq n < n_0+N : (i_n, \ldots, i_{n+d}) = \iota \right\} - p_{\iota, b} \right| \xrightarrow[N \to \infty]{} 0.
    \]
    Moreover, $b\mapsto p_{\iota, b}$ is continuous on fully irrational directions $\mathcal{I}$. 
\end{prop}

To prove this we build a framework for understanding the sequence $I$. Recall that it is obtained as follows: starting the laser at $L_0$, we move in the direction of $b$ and record the labels of the faces we hit. More precisely, let $\Hc\subset \RR^d$ be the union of all hyperplanes, and $T_b\colon\Hc\to \Hc$ be the (flow) map that moves $x\in \Hc$ in the direction of $b$ until it returns to $\Hc$. Then $I$ is the sequence of labels of the faces containing the orbit of $x$. 

For each pattern $\iota = (\iota_0, \dots, \iota_d)$ let $O_{\iota, b}\subset \Hc$ be the set of points $x$ such that the labels of the faces containing $x, T_bx, \dots, T_b^dx$ are precisely $\iota_0, \dots, \iota_d$. Our goal is to prove that the orbit of $x$ visits $O_{\iota, b}$ with a well-defined limiting frequency. Therefore statistical properties of $I$ correspond to dynamical properties of the system $(\Hc, T_b)$ (see \cref{rmk: quotient space X}).  

We aim to use the periodicity of the labeling modulo $Q^\vee$ to analyze this system. It turns out that $T_b$ can be decomposed into $|W|$ interlaced translations on $d-1$ dimensional tori. 

Fix $\alpha \in \Phi$ and consider the hyperplanes orthogonal to $\alpha$. Reflecting over $H_{\alpha}^{k}$ shows that the labeling on $H_{\alpha}^{k-1}$ and $H_{\alpha}^{k+1}$ are translates of each other, as shown in \cref{fig: orange hyperplanes}. Define
$$\Hc_\alpha^0 = \{ H_{\alpha}^{2k} : k \in \mathbb{Z} \}, \qquad \Hc_\alpha^1 = \{ H_{\alpha}^{2k+1} : k \in \mathbb{Z} \}.$$
For each family we define the \dfn{first return time} and \dfn{first return map} by
$$\tau_\alpha^i(x) := \min\left\{ n>0 : T_b^n(x) \in \Hc_\alpha^i \right\} \quad \text{and} \quad R_\alpha^i(x) := T_b^{\tau_\alpha^i(x)}(x).$$
Since the labeling is $Q^\vee$-invariant, it suffices to consider this map on the quotient $\TT_\alpha^i = \Hc_\alpha^i/(Q^\vee \cap \Hc_\alpha^i)$, which is
an $d-1$ dimensional torus ($Q^\vee\subset \RR^d$ has full rank). Because the hyperplanes in $\Hc_\alpha^0$ differ by translations of $\alpha^\vee$, which is orthogonal to each of them, the induced map $R_\alpha^i$ on $\TT_\alpha^i$ is simply a rotation by $v_\alpha^i = b-(b^\top \alpha^\vee)\alpha^\vee$, as illustrated in \cref{fig: torus rotation}.

\begin{figure}[H]
    \centering
    \begin{subfigure}[b]{0.60\textwidth}
        \includegraphics[width=\linewidth]{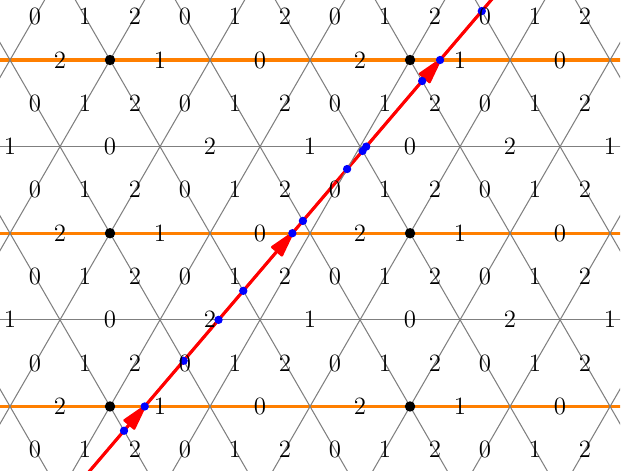}
        \caption{The orange hyperplanes $\Hc_\alpha^i$ are arranged at intervals of $\alpha^\vee$. Their labelings are translations of each other.}
        \label{fig: orange hyperplanes}
    \end{subfigure}
    \hfill
    \begin{subfigure}[b]{0.3\textwidth}
        \includegraphics[width=\linewidth]{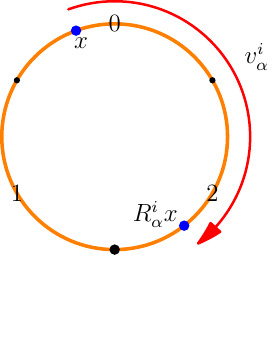}
        \caption{Modulo the lattice $Q^\vee$, $R_\alpha^i$ becomes a rotation in $\TT^{d-1}$.}
        \label{fig: torus rotation}
    \end{subfigure}

    \caption{Illustration of one of the interlaced rotations that forms $T_b$.}
    \label{fig: map on torus}
\end{figure}

The ergodicity of such maps is well studied. For a torus $ \TT^r := \RR^r/\ZZ^r$, let $\lambda_{\TT^r}$ denote its standard Lebesgue measure. 

\begin{theorem}[Kronecker--Weyl, \cite{EW11}]\label{thm: kronecker--weyl}
    Let $A\subset \TT^r$ be a Borel set, and $\alpha \in \RR^d$ a fully irrational vector. For every continuous $f: \TT^r \to \RR$ and $x\in \TT^r$ we have
    $$\frac{1}{N}\sum_{i=0}^{N-1} f(x+\alpha i) \xrightarrow[N\to \infty]{} \int f \, \mathrm{d}\lambda_{\TT^r}.$$
    Moreover, if $f = \1_A$ and $A$ has zero measure boundary, the convergence is uniform in $x$. 
\end{theorem}

We now justify that this result applies to the sets we are interested in. 

\begin{lemma} \label{lmm: regularity of Owb}
     Each set $O_{\iota, b}\subseteq \Hc$ is either empty or a union of open $d-1$ dimensional simplices, and in particular is Borel. Moreover, it varies continuously with $b\in S^{d-1}$.
\end{lemma}
\begin{proof}
    Fix a face $F\subset X$. For each point $x\in F \subset \RR^d$, we determine what set $O_{\iota, b}$ it lies in by drawing a line segment in the direction of $b$ that intersects $d$ faces, and looking at the labels of the faces the segment intersects. Consider the (finite) set of lattice points $P\in Q^\vee$ such that the segment from $P$ to $F$ in the direction of $b$ intersects at most $d$ faces. The collection $\mathcal{L}$ of all these lines partitions space into a finite number of disjoint parallelepipeds (e.g. by taking a triangulation). All points in the same parallelepiped lie in the same $O_{\iota, b}$.
    
    For the second part, note that the boundaries $F\cap \mathcal{L}$ vary continuously with $b$, and that the label of a face $F\cap \mathcal{P}$ only changes when it degenerates.
\end{proof}

Therefore, for every starting point $x$, its orbit under $R_\alpha^i$ will equidistribute along the subtorus $\TT_{\alpha, b}^i\subset \TT_\alpha^i$ defined by the rational relations in the coordinates of $v_\alpha^i$ (with respect to $Q^\vee$). When $b$ is fully irrational this is $\TT^{d-1}$, but it could be lower dimensional. In particular, it will visit the set $O_{\iota, b}\cap \TT_{\alpha, b}$ with frequency $p_{\alpha,\iota,b}^i := \lambda_{\mathbb{T}_{\alpha,b}^i} \left(\mathcal{O}_{\iota,b} \cap \mathbb{T}_{\alpha,b}^i \right)$.

As the laser travels in a straight line, it visits \(\Hc_\alpha^i\) with asymptotic frequency proportional to \(|b^\top\alpha^\vee|\). More precisely, after $t$ seconds, the laser returns $t|b^\top \alpha^\vee| + O(1)$ times to $\Hc_\alpha^i$, so
\[
\#\{n_0 \leq n < n_0+N : T_b^n(x) \in \mathcal{F}_\alpha^i \} = \frac{|b^\top \alpha^\vee|}{2 \sum_{\beta \in \Phi} |b^\top \beta^\vee|} \cdot N + O(1).
\]
Therefore, the limiting frequency with which the pattern \(\iota\) appears is
\begin{equation}\label{eq: limiting frequency}
    p_{\iota, b} = \frac{1}{2 \sum_{\beta \in \Phi} |b^\top \beta^\vee|} \sum_{\alpha \in \Phi} |b^\top \alpha^\vee|\left( p_{\alpha,\iota,b}^0 + p_{\alpha,\iota,b}^1 \right).
\end{equation}
Note that for a fixed $b$ and $L_0$, each $\TT_{\alpha, b}^i$ is independent of the starting point $T^{n_0}x$, so uniformity in $n_0$ follows from the uniform convergence in \cref{thm: kronecker--weyl}. For fully irrational $b$, from \cref{lmm: regularity of Owb} we see each frequency $p_{\alpha, \iota, b}^i$ is continuous. Hence \eqref{eq: limiting frequency} shows $p_{\iota, b}$ is too.

\begin{remark}[Compactifying $(\Hc, T_b)$] \label{rmk: quotient space X}
    The labeling is $Q^\vee$-invariant, so we can consider $T_b$ on the quotient $X:=\Hc/Q^\vee$, as shown in \cref{fig: dynamical system}. Then \cref{lmm: equidistribution of windows} is equivalent to a uniform Birkhoff ergodic theorem for $\1_{\pi(O_{\iota, b})}$. 
    \begin{figure}[H]
        \centering
        \includegraphics[width=0.4\linewidth]{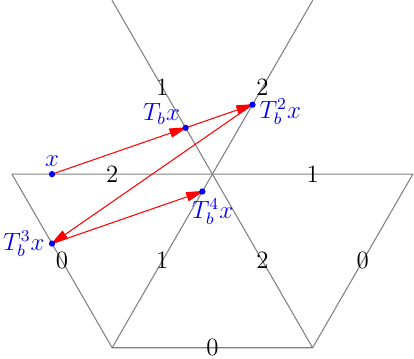}
        \caption{The dynamical system $(X, T_b)$ when $\widetilde{W} = \widetilde{A}_2$. The cutting sequence produced is $I = (2, 1, 2, 0, 1, \dots)$. }
        \label{fig: dynamical system}
    \end{figure}
    Our argument above shows there is a measure $\mu_b$ on $Y = \overline{\{x, T_bx, T_b^2x, \dots\}}$ such that $(Y, T_b, \mu_b)$ is uniquely ergodic, and thus $p_{\iota, b} = \mu_b(O_{\iota, b})$. On fully irrational directions $Y=X$ and $\mu_b$ is continuous. Both become singular as soon as there is a rational dependence. 
\end{remark}

\section{Consistent Growth of the Covariance Matrix}\label{sec: covariance limit}

In this section, we prove the conditional covariance of $(X_n)_{n\geq 0}$ grows consistently. 

\begin{prop}\label{lmm: limiting sigma} 
    There exists $\sigma := \sigma_{b, p}> 0$ such that for every $n_0\geq 0$
    $$\frac{1}{N} \EE\left[(X_{N+n_0}-X_{n_0})(X_{N+n_0}-X_{n_0})^\top\mid \Fc_{n_0}\right] \xrightarrow[N\to \infty]{} \sigma^2I_d.$$
    The convergence is uniform in $n_0$. Moreover, $\sigma_{b, p}$ is jointly continuous on $\mathcal{I}\times (0, 1)$. 
\end{prop}

We equip the space of $d\times d$ matrices $\mat_{d\times d}(\RR)$ with the norm
$\|A\|_{\op}:= \sup_{\|x\|=1} \|Ax\|$.

\subsection{Decomposition into step interactions}

We break $(X_n)_{n\geq 0}$ into the steps $\Delta X_n$, so
$$X_{N+n_0} = X_{n_0}+\Delta X_{n_0} + \Delta X_{n_0+1} + \cdots + \Delta X_{n_0+N-1}.$$ 
Expanding the covariance, we obtain
\begin{equation}\label{eq : big sum}
    \frac{1}{N}\EE\left[(X_{N+n_0}-X_{n_0})(X_{N+n_0}-X_{n_0})^\top\right] = \sum_{m=-(N-1)}^{N-1} \Sigma_{m; \,N, n_0}
\end{equation}
where
$$\Sigma_{m; \,N, n_0}:=\frac{1}{N}\sum_{n_0\leq n, \, n+m< N+n_0}\EE\left[\Delta X_{n} \Delta X_{n+m}^\top \mid \Fc_{n_0}\right].$$
For each $m \in \ZZ$, we claim there exists a matrix $\Sigma_m$ such that $\Sigma_{m; \,N, n_0} \to \Sigma_m$ as $N\to \infty$. Since $\Sigma_{-m; \,N, n_0} = \Sigma_{m; \,N, n_0}^\top$, it suffices to consider $m\geq 0$. 

The expectation of each interaction matrix $\Delta X_n \Delta X_{n+m}^\top$ depends on the distribution of $\Delta X_n$, and the conditional distribution of $\Delta X_{n+m}$ given $\Delta X_n$. The latter depends only on the \dfn{window} $\iota=(i_n$, \dots, $i_{n+m})$, and we may write
\begin{equation}\label{eq: sigma m, N, n0}
    \Sigma_{m; \, N, n_0} = \sum_{\iota \in \{0, 1, \dots, d\}^{m+1}} \frac{1}{N}\sum_{\substack{n_0\leq n < N-m + n_0 \\ (i_{n}, \dots, i_{n+m}) = \iota}} \EE\left[\Delta X_n \Delta X_{n+m}^\top \mid \Fc_{n_0}\right].
\end{equation}
\cref{lmm: equidistribution of windows} implies the fraction of terms in each interior sum converges to $p_\iota$. We now use the properties of the limiting distribution of $\Delta X_n$ from \cref{sec: mixing} to access each summand.

\subsection{Analysis of a single interaction term}
Using \eqref{prop: DXn}, the fact that $s_i^{-1} = s_i$ for each $i$, and moreover that $\rho(\pi(s_i))$ is a reflection matrix and thus symmetric, we can write
\begin{equation} \label{eq: big}
    \begin{split}
        \Delta X_n \Delta X_{n+m}^\top &=\left[\eps_n (\rho\circ \pi)\left(s_{i_1}^{\eps_1}\cdots s_{i_{n-1}}^{\eps_{n-1}}\right)\beta_{i_n}\right]\left[\eps_{n+m} (\rho\circ \pi)\left(s_{i_1}^{\eps_1}\cdots s_{i_{n+m-1}}^{\eps_{n+m-1}}\right)\beta_{i_{n+m}}\right]^\top\\ &=(\rho\circ \pi)\left(w_{n-1}\right)\left[ \eps_n\eps_{n+m}\beta_{i_n}\beta_{i_{n+m}}^\top (\rho\circ \pi)\left(s_{i_{n+m}}^{\eps_{n+m}}\cdots s_{i_{n+1}}^{\eps_{n+1}}\right)\right](\rho\circ \pi)\left(w_{n-1}\right)^{\top} \\
        &= U_{n-1} A_{i_n, \dots, i_{n+m}} U_{n-1}^{-1},
    \end{split}
\end{equation}
where we define the random matrices
$$U_{n-1}:= (\rho\circ \pi)\left(s_{i_1}^{\eps_1}\cdots s_{i_{n-1}}^{\eps_{n-1}}\right), \qquad A_{i_n, \dots, i_{n+m}}:=\eps_n\eps_{n+m}\beta_{i_n}\beta_{i_{n+m}}^\top (\rho\circ \pi)\left(s_{i_{n+m}}^{\eps_{n+m}}\cdots s_{i_{n+1}}^{\eps_{n+1}}\right).$$
The group $\GL(d, \RR)$ acts on $d\times d$ matrices via the \dfn{adjoint representation}
\begin{align*}
    \Ad\colon \GL(d, \RR) &\to \End(\mat_{d\times d}(\RR)) \\
    U &\mapsto \left(\Ad_U\colon X\mapsto UXU^{-1}\right).
\end{align*}
Since $\eps_{n}, \dots, \eps_{n+m-1}$ are independent from $\eps_1, \dots, \eps_{n-1}$, taking the expectation of \eqref{eq: big} gives
\begin{equation*}\label{eq: adjoint}
    \EE_{\eps_1, \dots, \eps_{n+m-1}}\left[\Delta X_n \Delta X_{n+m}^\top \mid \Fc_{n_0}\right] = \EE_{\eps_1, \dots, \eps_{n-1}}\left[\Ad_{U_{n-1}}\left(\EE_{\eps_n, \dots, \eps_{n+m-1}}[A_{i_n, \dots, i_{n+m}}]\right) \mid \Fc_{n_0}\right].
\end{equation*}

\begin{lemma}[Averaging over $W$] \label{lmm: EAdU}
    Fix $X\in \mat_{d\times d}(\RR)$. For every $n\geq n_0$ we have
    $$\EE\left[\Ad_{U_{n-1}} (X) \mid \Fc_{n_0}\right] = \frac{1}{d} \Tr(X) \cdot I_d + O(c^{n-n_0}).$$
\end{lemma}
\begin{proof}
    In the vector space $\End(\mat_{d\times d}(\RR))$ we have 
    \begin{equation}\label{eq: Sigma dN}
        \begin{split}
            \EE\left[\Ad_{U_{n-1}} \mid \Fc_{n_0}\right] &= \sum_{w\in W} \Prob[\pi(w_{n}) = w \mid \Fc_{n_0}]\cdot \Ad_{\rho(w)} \\
        &= \sum_{w\in W} \frac{1}{|W|}\Ad_{\rho(w)} + \sum_{w\in W} \left(\Prob[\pi(w_{n}) = w \mid \Fc_{n_0}] - \frac{1}{|W|}\right)\Ad_{\rho(w)}.
        \end{split}
    \end{equation}
    By \cref{lmm: quick mixing} we know the second (error) term has operator norm $O(c^{n-n_0})$. Thus, it suffices to understand the first term, which we denote by $\Ad_{W}$. Given $X$, define
    \begin{align*}
        T_X:=\Ad_{W}(X) = \frac{1}{|W|} \sum_{w\in W} \rho(w) X \rho(w)^{-1}.
    \end{align*}
    We claim $T_X = \Tr(X)/r\cdot I_d$. Clearly $\rho(y) T_X = T_X\rho(y)$ for every $y\in W$, so $T_X \colon \CC^d\to \CC^d$ is a homomorphism of representations of $\rho$ to itself. Since $\rho$ is irreducible, it follows from Schur's lemma \cite[Lemma~4]{S77} that $T_X = \lambda I$ for some $\lambda \in \CC$. Moreover, 
    $$\Tr(\Ad_{W}(X)) = \frac{1}{|W|}\sum_{w\in W} \Tr(\rho(w)X\rho(w)^{-1}) = \frac{1}{|W|}\sum_{w\in W} \Tr(X\rho(w)^{-1}\rho(w)) = \Tr(X).$$
    Thus $\lambda = \Tr(X)/d$, so $\Ad_{W}(X) = 1/d\cdot \Tr(X) \cdot I_d$. 
\end{proof}

\subsection{Convergence to \texorpdfstring{$\Sigma$}{Sigma}}

We now have all the pieces to put together. For each $m\geq 0$ define
\begin{equation*}
    \Sigma_m := \frac{1}{d}\sum_{\iota \in \{0, 1, \dots, d\}^{m+1}} p_{\iota} \Tr(\EE[A_\iota])\cdot I_d.
\end{equation*}
\begin{lemma} \label{prop: SdN}
    We have $\Sigma_{m; \, N, n_0} \to \Sigma_m$ as $N\to \infty$, uniformly in $n_0$.
\end{lemma}
\begin{proof}
    Combining \eqref{eq: sigma m, N, n0} and \eqref{eq: adjoint}, and using \cref{lmm: EAdU} and \cref{lmm: equidistribution of windows}
    \begin{align*}
        \Sigma_{m; \, N, n_0}
        &= \sum_{\iota \in \{0, 1, \dots, d\}^{m+1}} \frac{1}{N}\sum_{\substack{n_0\leq n < N-m+n_0 \\ (i_n, \dots, i_{n+m}) = \iota}} \EE\left[\Ad_{U_{n-1}}(\EE[A_\iota]) \mid \Fc_{n_0}\right] \\
        &=\sum_{\iota \in \{0, 1, \dots, d\}^{m+1}} \frac{1}{N}\sum_{\substack{n_0\leq n < N-m+n_0 \\ (i_n, \dots, i_{n+m}) = \iota}} \left(\Tr(\EE[A_\iota]) + O(c^{n-n_0})\right)\cdot I_d \\
        &= \sum_{\iota \in \{0, 1, \dots, d\}^{m+1}} (p_\iota + o(1))\cdot\frac{1}{d}\Tr(\EE[A_\iota])\cdot I_d +  O\left(N^{-1}\right),
    \end{align*}
    which converges to $\Sigma_m$. The $o(1)$ errors are uniform in $n_0$, giving uniform convergence.
\end{proof}

\begin{lemma} \label{lmm: domination}
    We have 
    $$\frac{1}{N} \EE\left[(X_{N+n_0}-X_{n_0})(X_{N+n_0}-X_{n_0})^\top\mid \Fc_{n_0}\right] \xrightarrow[N\to \infty]{} \Sigma$$
    uniformly in $n_0$, where $\Sigma$ is defined by the absolutely convergent sum
    \begin{equation}\label{eq: C}
        \Sigma := \cdots + \Sigma_{-2} + \Sigma_{-1} + \Sigma_0 + \Sigma_1 + \Sigma_2 + \cdots.
    \end{equation}
    Moreover, it converges uniformly in $b$ and locally uniform in $p$. 
\end{lemma}
\begin{proof}
    From~\cref{prop: independence} we know that for any $n\geq n_0$ and $m\geq 0$
    $$\left\|\EE\left[\Delta X_n \Delta X_{n+m}^\top \mid \Fc_{n_0}\right]\right\|_{\op} \leq \EE\left[\left\|\Delta X_n\right\|\cdot  \|\EE[\Delta X_{n+m}\mid\Delta X_n, \Fc_{n_0}]\|\right] \leq Cc^m,$$
    so $(Cc^{|m|})_{m\in \ZZ}$ dominates $(\Sigma_{m; \, N, n_0})_{m\in \ZZ}$. Since the latter converges absolutely, the first conclusion follows from \cref{prop: SdN} and the dominated convergence theorem. The second follows from the fact that both $c$ and $C$ are independent of $b$ and locally bounded functions of $p$.
\end{proof}

Critically, this does \textit{not} yield $\Sigma \succ 0$; each individual $\Sigma_m$ need not be positive semidefinite. 

\begin{proof}[Proof of \cref{lmm: limiting sigma}]
    By \eqref{eq: C} and \cref{lmm: EAdU}, $\Sigma$ is a scalar matrix:
    \[
        \Sigma = \frac{1}{d} \sum_{m\ge 0} \sum_{\iota\in \{0,\dots,d\}^{m+1}} p_\iota \Tr(\EE[A_\iota])\, I_d = \sigma^2 I_d.
    \]
    Positivity $\sigma > 0$ follows from \cref{prop: fn grows}. Continuity for $(b,p) \in \mathcal{I}\times(0, p)$ follows from uniform convergence and continuity of each $\Sigma_m$.
\end{proof}

\section{Proof of \texorpdfstring{\cref{thm: main theorem}}{main theorem}} \label{sec: higher moments}

\subsection{Higher Moments}
In \cref{sec: covariance limit} we established the convergence of the second moment of $(X_n-X_0)/\sqrt{n} \in \RR^d$.  
More generally, for any $k \geq 1$, the $k$-th \dfn{moment tensor} of a random vector $Y \in \RR^d$ is the tensor $\EE[Y^{\otimes k}] \in (\RR^d)^{\otimes k}$, defined by the evaluation rule
\[
\mathbb{E}\left[Y^{\otimes k}\right](v_1,\dots,v_k) \;=\; \mathbb{E}\left[ \prod_{i=1}^k (Y^\top v_i) \right],
\quad v_1,\dots,v_k \in \RR^d.
\]
Note that for $k=2$ this reduces to the usual bilinear form $v_1^\top \EE[YY^\top] v_2$. We equip the tensor space $(\RR^d)^{\otimes k}$ with the \dfn{operator norm}
$$\|T\|_{\mathrm{op}} :=\sup_{\|v_1\|=\cdots=\|v_k\|=1} \, |T(v_1,\dots,v_k)|.$$

We now establish that all moments of $X_n/\sqrt{n}$ to those of a Gaussian $X\sim \mathcal{N}(0, \sigma^2 I_d)$ (under the operator norm). As a consequence of Isserlis' theorem \cite{I18}, these are
\begin{align*}
    \mathbb{E}[X^{\otimes (2k+1)}] = 0, \qquad \mathbb{E}[X^{\otimes 2k}] = \sigma^{2k}
    \sum_{\pi \in \mathcal{P}_{2k}} \,
    \sum_{\substack{1\leq n_1,\dots,n_{2k}\leq r \\ n_i = n_j \,\forall (i,j)\in \pi}}
    e_{n_1} \otimes \cdots \otimes e_{n_{2k}},
\end{align*}
where $(e_i)_{1\leq i \leq d}$ is the standard basis of $\RR^d$, and $\mathcal{P}_{2k}$ the set of perfect matchings of $\{1,\dots,2k\}$. 

\begin{lemma} \label{lmm: higher moments converge} 
    Let $X\sim \mathcal{N}(0, \sigma^2 I_d)$. Then for every $k > 2$ we have
    $$\EE\left[\left(\frac{X_N-X_0}{\sqrt{N}}\right)^{\otimes k}\right] \xrightarrow[N\to \infty]{} \EE[X^{\otimes k}].$$
\end{lemma}

Convergence of all moments implies $(X_n-X_0)/\sqrt{n} \xrightarrow[]{\mathcal{D}} \mathcal{N}(0, \sigma^2 I_d)$, proving \cref{thm: main theorem}.

\subsection{Proof of \texorpdfstring{\cref{lmm: higher moments converge}}{convergence of higher moments}}

Writing $X_N = X_0+ \Delta X_0 + \cdots + \Delta X_N$ we obtain
\begin{equation}\label{eq: expand}
    \frac{1}{N^{\frac{k}{2}}} \EE[(X_N-X_0)^{\otimes k}]=\frac{1}{N^{\frac{k}{2}}}\sum_{0\leq n_1, \dots, n_{k} \leq N}\EE[\Delta X_{n_{1}}\otimes \cdots \otimes \Delta X_{n_{k}}].
\end{equation}
We show that, in the limit, this expression is dominated by the second-order interaction terms. That is, by sequences that can be partitioned into pairs of close indices, as in \cref{fig:blocks}. To make this precise, fix some threshold $\delta$. Build a graph $G_{n_1, \dots, n_k}^\delta$ on $1, \dots, k$ by connecting $i\sim j$ whenever $|n_i-n_j|\leq \delta$. We call a connected component in this graph a \dfn{block}.

\begin{figure}[H]
    \centering
    \includegraphics[width=0.8\linewidth]{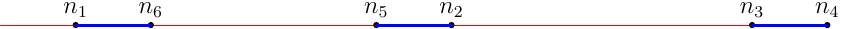}
    \caption{A sequence $n_1, \dots, n_6$ whose graph has blocks $\{1, 6\}$, $\{2, 5\}$, $\{3, 4\}$.}
    \label{fig:blocks}
\end{figure}

\begin{lemma}\label{prop: singletons}
    If $G_{n_1, \dots, n_k}^\delta$ contains a block with only one element, then
    $$\|\EE[\Delta X_{n_1}\otimes \cdots \otimes \Delta X_{n_{k}}]\| \leq O(c^\delta).$$
\end{lemma}
\begin{proof}
    Applying an isometry $T_\sigma : v_1\otimes \cdots \otimes v_{k} \mapsto v_{\sigma(1)}\otimes \cdots \otimes v_{\sigma(k)}$ we may assume without loss of generality that $n_1\leq n_2\leq \cdots \leq n_k$. By the law of iterated expectation
    \begin{align*}
        &\left\|\EE\left[\Delta X_{n_1}\otimes \cdots \otimes \Delta X_{n_{k}}\right]\right\|_{\op} = \left\|\EE\left[\Delta X_{n_1} \otimes \cdots \otimes \EE[\Delta X_{n_k}\mid \Delta X_{n_{k-1}}\vee \dots\vee \Delta X_{n_1}]\right]\right\|_{\op} \\
        &\leq \prod_{i=1}^k \left\|\EE\left[\Delta X_{n_i} \mid \Delta X_{n_{i-1}}\vee \dots\vee\Delta X_{n_1}\right]\right\|.
    \end{align*}
    All terms are at most $\|\Delta X_{n_i}\| \leq 1$. At the singleton, \cref{prop: independence} gives a bound of $O(c^\delta)$.
\end{proof}

We now show most sequences without singletons have exactly $\lfloor k/2\rfloor $ blocks.

\begin{lemma}\label{prop: few blocks}
    There are $O_k(N^\ell \delta^{k-\ell})$ sequences $1\leq n_1, \dots, n_k\leq N$ with $\ell\leq k/2$ blocks.
\end{lemma}
\begin{proof}
    Fix a partition of $\{1, \dots, k\}$ into $\ell$ blocks. Choose the positions of $n_1, \dots, n_{2k}$ one at a time, producing this partition. For each block there are at most $N$ choices for the first element, while each subsequent element must lie within distance $\delta\cdot k$. This shows there are at most $N^\ell (\delta k)^{k-\ell}$ such sequences. Finally, the number of partitions is constant in $k$.
\end{proof}

Suppose $k = 2s+1$. If there are no singletons, then the number of blocks is at most $s$. Therefore, \cref{prop: singletons} and \cref{prop: few blocks} imply
$$N^{-s-\frac{1}{2}} \EE\left[(X_N-X_0)^{\otimes (2s+1)}\right] = N^{-s-\frac{1}{2}} \cdot N^{2s+1}O(c^\delta) + N^{-s-\frac{1}{2}}\cdot O_k(N^s\delta^{s+1}) = o(1),$$
by choosing $\delta = \Theta(\log N)$ for a sufficiently large constant.

If $k=2s$, the contribution
of sequences whose graphs are not in $\mathcal{P}_{2s}$ is upper-bounded by
\begin{equation}\label{eq: only 2 blocks}
    N^{-s}\cdot N^{2s}O(c^\delta) + N^{-s}\cdot O_k(N^{s-1}\delta^{s+1}) = O_k(N^s c^\delta + N^{-1}\delta^{s+1}) = o(1),
\end{equation}
where we choose $\delta = \Theta(\log N)$ with a sufficiently large constant. 

\begin{lemma}\label{prop: key}
    Suppose $\delta = \Theta(\log N)$. Then for every $\pi\in \mathcal{P}_{2s}$, as $N\to \infty$ we have
    $$\frac{1}{N^{s}}\sum_{\substack{0\leq n_1, \dots, n_{2s}\leq N \\ G_{n_1, \dots, n_{2s}}^\delta= \pi}} \EE\left[\Delta X_{n_{1}}\otimes \cdots \otimes \Delta X_{n_{2s}}\right] \to  \sigma^{2s}\sum_{\substack{1\leq n_1, \dots, n_{2s}\leq r \\ n_i = n_j \, \forall (i, j)\in \pi}} e_{n_1}\otimes \cdots \otimes e_{n_{2s}}.$$
\end{lemma}

The difference between matchings is purely notational, since $(T_\sigma)_{\sigma \in S_{2s}}$ acts transitively on $\mathcal{P}_{2s}$. Thus it suffices to consider $\pi = \pi_0 := \{(1, 2), \dots, (2s-1, 2s)\}$. Suppose the blocks are
$$(n_1, n_2) = (m_1, m_1+\delta_1),  \quad (n_3, n_4) = (m_2, m_2+\delta_2), \quad \dots, \quad (n_{2s-1}, n_{2s}) = (m_{s}, m_s+\delta_s), $$
with $|\delta_i| \leq \delta$ for all $i$, and all edges disconnected. Let $T_{n, \delta} := \Delta X_n \otimes \Delta X_{n+\delta}$, so
$$\EE[\Delta X_{m_1}\otimes \Delta X_{m_1+\delta_1}\otimes \cdots \otimes \Delta X_{m_s}\otimes \Delta X_{m_s+\delta_s}] = \EE[T_{m_1, \delta_1}\otimes \cdots \otimes T_{m_s, \delta_s}].$$
We now show we can break this expression apart into its order-2 terms, with small error. 
\begin{prop}\label{prop: structure}
In the situation above we have
$$\EE[T_{m_1, \delta_1}\otimes \cdots \otimes T_{m_s, \delta_s}] = \EE[T_{m_1, \delta_1}]\otimes \cdots \otimes \EE[T_{m_s, \delta_s}] + O_s(c^\delta).$$
\end{prop}
\begin{proof}
    We induct on $s$. Without loss of generality the blocks appear in increasing order. Suppose $s > 1$ and let $\Gc:= \sigma(T_{m_1, \delta_1}, \dots, T_{m_{s-1}, \delta_{s-1}})$. By the law of iterated expectation
    $$\EE[T_{m_1, \delta_1}\otimes \cdots \otimes  T_{m_{s-1}, \delta_{s-1}}\otimes T_{m_s, \delta_s}]= \EE\left[T_{m_1, \delta_1}\otimes \cdots \otimes T_{m_{s-1}, \delta_{s-1}} \otimes \EE[T_{m_s, \delta_s} \mid \Gc ]\right].$$
    Since $\Gc\subseteq \Fc_{m_{s-1}+\delta_{s-1}}$ and $\min(m_s, m_s+\delta_s) > m_{s-1} + \delta_{s-1} + \delta$, \cref{lmm: EAdU} yields
    $$\EE[T_{m_s, \delta_s} \mid \Gc ] = \EE[T_{m_s, \delta_s}] + O(c^\delta).$$
    Because $\|\EE[T_{m_1, \delta_1}\otimes \cdots \otimes T_{m_{s-1}, \delta_{s-1}}]\|\leq 1$, the inductive step is complete. 
\end{proof}

\begin{proof}[Proof of \cref{prop: key}]
    As above, suppose $\pi = \pi_0$. By \cref{prop: structure}
    \begin{align*}
    &\frac{1}{N^{s}}\sum_{\substack{0\leq n_1, \dots, n_{2s}\leq N \\ G_{n_1, \dots, n_{2s}}^\delta= \pi_0}} \EE[\Delta X_{n_{1}}\otimes \cdots \otimes \Delta X_{n_{2s}}] = \frac{1}{N^s} \sum_{\substack{0\leq m_1, \dots, m_{s}\leq N \\ |\delta_1|, \dots, |\delta_{s}|\leq \delta}} \EE[T_{m_1, \delta_1} \otimes \cdots \otimes T_{m_{s}, \delta_{s}}] + o(1) \\ 
    &=\frac{1}{N^s} \sum_{\substack{0\leq m_1, \dots, m_{s}\leq N \\ |\delta_1|, \dots, |\delta_{s}|\leq \delta}} \EE[T_{m_1, \delta_1}] \otimes \cdots \otimes \EE[T_{m_{s}, \delta_{s}}] + \delta^s O_s(c^\delta) + o(1),
    \end{align*}
    where the $o(1)$ accounts for sequences with \textit{fewer} than $s$ blocks (c.f. \eqref{eq: only 2 blocks}). We factor
    $$\frac{1}{N^s} \sum_{\substack{0\leq m_1, \dots, m_{s}\leq N \\ |\delta_1|, \dots, |\delta_{s}|\leq \delta}} \EE[T_{m_1, \delta_1}] \otimes \cdots \otimes \EE[T_{m_{s}, \delta_{s}}] =  \left(\frac{1}{N}\sum_{\substack{|\delta'| \leq \delta \\ 0\leq n \leq N}} \EE[T_{n, \delta'}]\right)^{\otimes s}.$$
    Comparing to \cref{lmm: limiting sigma}, the term in the right-hand side is nearly 
    $$\frac{1}{N}\sum_{\substack{\delta'\in \ZZ\\ 0\leq n \leq N}} \EE[T_{n, \delta'}]=\frac{1}{N} \EE[X_n\otimes X_n] = \sigma^2 \sum_{i=1}^r e_i\otimes e_i + o(1).$$ 
    There is an additional error of magnitude $O(c^\delta)$ from only adding over $|\delta'| \leq \delta$. Finally,
    $$\left(\sigma^2\sum_{i=1}^n e_i\otimes e_i\right)^{\otimes s}  =  \sigma^{2s}\sum_{\substack{1\leq n_1, \dots, n_{2s}\leq r \\ n_{2i-1} = n_{2i} \, \forall i}} e_{n_1}\otimes \cdots \otimes e_{n_{2s}},$$
    which yields the desired conclusion. 
\end{proof}

\section{Proof of \texorpdfstring{\cref{thm: main theorem 2}}{second main theorem}} \label{sec: martingale approximation}

\subsection{The Martingale Approximation}
In this section, we show that the scaled trajectories of the random billiard walk converge to Brownian motion. To do this, we need a global perspective on $(X_n)_{n\geq 0}$. We group consecutive steps into larger \dfn{moving steps} of sizes $a_1$,~$a_2$,~$\dots$. In order to decorrelate the displacement during moving steps, we insert \dfn{mixing steps} of lengths $b_1$,~$b_2$,~$\dots$ between them, as illustrated in~\cref{fig:martingale approx}. 

\begin{figure}[H]
    \centering
    \includegraphics[width=0.9\linewidth]{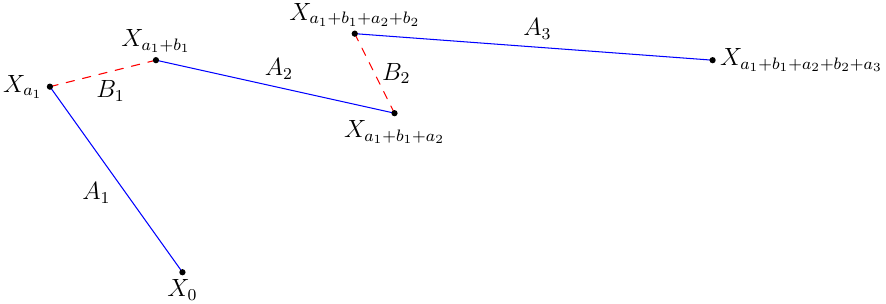}
    \caption{The definition of $(a_n)_{n\geq 0}$ and $(b_n)_{n\geq 0}$. Moving steps are drawn in solid blue, and mixing steps are drawn in dashed red.}
    \label{fig:martingale approx}
\end{figure}

For convenience, set $s_n:= a_1+b_1 + \cdots + a_n+b_n$. Then the moving and mixing steps are 
$$A_i := X_{s_{i-1} + a_i} - X_{s_{i-1}} \quad \text{and} \quad B_i:= X_{s_i} - X_{s_{i-1} + a_i},$$
respectively. We approximate $X_{s_n} -X_0 = A_1 + B_1 + \cdots + A_n + B_n$ with the martingale
\begin{equation}\label{eq: def martingale}
    M_n := \sum_{i=1}^n (A_i - \EE[A_i\mid\mathcal{F}_{s_{i-1}-b_{i-1}}]).
\end{equation}
It is adjusted to the filtration $\mathcal{G}_i := \mathcal{F}_{s_{i-1}-b_{i-1}}$, and satisfies $\EE[M_n\mid \Gc_{n-1}] = 0$ by construction.

The mixing steps should be sufficiently small to not have an effect on the overall displacement, but large enough to guarantee the near independence of moving steps. To that end we set 
$$a_n := \lceil n^{1/3}\rceil \quad \text{and} \quad b_n := \lceil 2\log_{1/c}(n+1)\rceil.$$
(Any exponent in $(0, 1/2)$ would work; we take $1/3$ for concreteness). Then $s_n = \Theta(n^{4/3})$, $\|B_i\|\leq i^{1/3} \ll a_i$, and also the correction terms are small (and summable):
\begin{equation}\label{eq: mixing steps}
    \|\EE[A_i\mid\mathcal{F}_{s_{i-1}-b_{i-1}}]\|\leq Cc^{b_{i-1}} \leq \frac{C}{i^2}.
\end{equation}

Define the rescaled trajectory $M^{(n)}$ by linear interpolation, as in \eqref{eq: rescaled trajectory}. By \cref{lmm: Slutsky}, to prove \cref{thm: main theorem 2} suffices to establish the following two results.

\begin{lemma}\label{lmm: martingale convergence}
    The trajectories $M^{(n)}$ converge to a Brownian motion with covariance $\sigma^2 I_d$. 
\end{lemma}

\begin{lemma}\label{lmm: good approximation}
    The approximation error $\|M^{(n)} -X^{(n)}\|_{\infty}$ converges to 0 in probability. 
\end{lemma}

\subsection{Proof of \texorpdfstring{\cref{lmm: martingale convergence}}{Martingale Convergence}}

We apply the Lindeberg–Feller martingale central limit theorem. This is a standard result for real valued martingales and can be straightforwardly adapted to $\RR^d$--martingales using the Cramér--Wold device, as in \cite{HH14}.

\begin{theorem} \label{thm: vector martingale clt}
    Let $(M_n)_{n \geq 1}\subset \RR^d$ be a martingale with respect to a filtration $(\mathcal{F}_n)_{n\geq 1}$. 
    Define $Z_n := M_{n}-M_{n-1}$. Suppose $(s_n)_{n\geq 0}$ is a deterministic sequence going to infinity such that
    \begin{itemize}
        \item For some $\Sigma \in \mat_{d\times d}(\RR)$
        \begin{equation}\label{eq: convergence to Sigma}
        \frac{1}{s_n}\sum_{i=1}^n \EE\left[Z_iZ_i^\top \mid \mathcal{F}_{i-1}\right]\xrightarrow[n\to \infty]{p} \Sigma.
        \end{equation}
        \item For every $\varepsilon > 0$ 
    \begin{equation}\label{eq: lindeberg}
        \frac{1}{s_n}\sum_{i=1}^n \EE\left[\|Z_i\|^2 \cdot \mathbf{1}_{\{\|Z_i\| > \varepsilon \sqrt{s_n} \}} \mid \mathcal{F}_{i-1} \right] \xrightarrow[n\to \infty]{p} 0.
    \end{equation}
    \end{itemize}    
    Then $M^{(n)}_t = \frac{1}{\sqrt{s_n}}M_{nt}$ converges to a Brownian motion with covariance $\Sigma$. 
\end{theorem}

The Lindeberg condition \eqref{eq: lindeberg} holds because the martingale increments $\|M_i - M_{i-1}\| \leq a_i$ are sufficiently small: given any \( \varepsilon > 0 \), there exists \( n_\eps \in \mathbb{N} \) such that for all \( n \geq n_\eps \)
\[
\|M_i - M_{i-1}\| \leq i^{1/3} <  \varepsilon \sqrt{s_n} = \Theta(n^{2/3})
\]
for each $i\leq n$. 
In this case, all the indicators in~\eqref{eq: lindeberg} vanish and the entire sum is zero; in particular, it converges to zero in probability.

We now verify \eqref{eq: convergence to Sigma} holds. For each $i$ we have $M_i - M_{i-1} = A_i - \EE[A_i\mid \Gc_{i-1}]$, so
\begin{align*}
    &\EE\left[(M_i-M_{i-1})(M_i-M_{i-1})^\top \mid \Gc_{i-1}\right]\\
    &= \EE[A_iA_i^\top \mid\Gc_{i-1}] - \EE\left[A_i\EE[A_i\mid \Gc_{i-1}]^\top \mid \Gc_{i-1}\right] -\EE\left[\EE[A_i\mid \Gc_{i-1}]A_i^\top \mid \Gc_{i-1}\right]+ \|\EE[A_i\mid \Gc_{i-1}]\|^2. 
\end{align*}
Using $\|A_i\|\leq i^{1/3}$ and \eqref{eq: mixing steps} we estimate
\begin{align*}
    &\left\|\frac{1}{s_n} \sum_{i=1}^n \EE\left[(M_i-M_{i-1})(M_i-M_{i-1})^\top \mid \Gc_{i-1}\right] - \frac{1}{s_N}\sum_{i=1}^n \EE[A_iA_i^\top \mid\Gc_{i-1}] \right\|_{\op} \\
    &\leq \frac{1}{s_n} \sum_{i=1}^n \| \EE\left[(M_i-M_{i-1})(M_i-M_{i-1})^\top \mid \Gc_{i-1}\right] - \EE[A_iA_i^\top \mid\Gc_{i-1}]\|_{\op}\\
    &\leq \frac{1}{s_n} \sum_{i=1}^n \left(2 i^{1/3}\cdot \frac{1}{i^2} + \frac{1}{i^4}\right).
\end{align*}
The series converges, so the difference goes to 0. Thus it suffices to study the simpler sum involving $A_i$. Since  $a_i\to \infty$ and the convergence in \cref{lmm: limiting sigma} is uniform, we know
$$\frac{1}{a_i}\EE\left[A_iA_i^\top \mid\Gc_{i-1}\right]=\frac{1}{a_i}\sum_{i=1}^n \EE\left[(X_{s_{i-1} + a_i} - X_{s_{i-1}})(X_{s_{i-1} + a_i} - X_{s_{i-1}})^\top\mid \Gc_{i-1}\right]\to \sigma^2 I_d.$$
Therefore, a Césaro argument shows
$$\frac{\sum_{i=1}^n \EE\left[A_iA_i^\top\right]}{\sum_{i=1}^n a_i} \to\sigma^2 I_d.$$
The conclusion follows from noting that $\frac{1}{s_n} \sum_{i=1}^n a_i \to 1$ (since $b_i\ll a_i$).

\subsection{Proof of \texorpdfstring{\cref{lmm: good approximation}}{Lemma 2.1}} 

For any $t\in [0, 1]$ there is some $k$ with $s_{k} \leq nt\leq s_{k+1}$. Then $|nt-s_k|\leq k^{1/3}\leq n^{1/3}$, so
$$\|X_{nt} - M_{nt}\|\leq \|X_{nt} - X_{s_k}\| + \|X_{s_k} - M_k\| + \|M_k-M_{nt}\| \leq O(n^{1/3}) + \|X_{s_k}-M_k\|,$$
and hence
\begin{align*}
    \left\|X^{(s_n)} - M^{(n)}\right\|_\infty &=\sup_{t\in [0, 1]} \left\|\frac{1}{\sqrt{s_n}} X_{s_n t} - \frac{1}{\sqrt{s_n}}M_{nt}\right\| \\
    &\leq \frac{O(n^{1/3}) + \max_{1\leq k\leq n} \|X_{s_k}-M_k\|}{\sqrt{s_n}} \\
    &= o(1) + \frac{1}{\sqrt{s_n}}\max_{1\leq k\leq n} \|X_{s_k}-M_k\|.
\end{align*}
Therefore, it suffices to prove a large deviations estimate: with high probability, $M_{k}$ is close to $X_{s_k}$ for every $k\leq n$. From the definition of $(M_n)_{n\geq 0}$ in~\eqref{eq: def martingale}
\begin{equation} \label{eq: difference}
    \begin{split}
        X_{s_k} - M_k = \sum_{i=1}^k \EE[A_i\mid \mathcal{F}_{s_{i-1}-b_{i-1}}] + \sum_{i=1}^k B_{i}.
    \end{split}
\end{equation}
The first error term corresponds to $(X_n)_{n\geq 0}$ not being a perfect martingale. From \eqref{eq: mixing steps}
\begin{equation}\label{eq: total mixing error}
    \left\|\sum_{i=1}^n \EE\left[A_i\mid \mathcal{F}_{s_{i-1}-b_{i-1}}\right] \right\| \leq C\sum_{i=1}^n c^{b_{i-1}} < \sum_{i=1}^\infty \frac{C}{i^2},
\end{equation}
which yields a uniform finite bound.

\begin{remark}
    Without mixing episodes, the upper bound in \eqref{eq: mixing steps} would be $\Theta(N)$ which, critically, is larger than the maximum error $o(\sqrt{s_n}) = o(n^{2/3})$ we can tolerate. 
\end{remark}

Now, our goal is to estimate the probability that $B_1+\cdots + B_k$ is small for every $k\leq n$. For $i\neq j$, $B_i$ and $B_j$ are far apart and almost independent, so the situation is similar to that in Kolmogorov's maximal inequality. We use a vector variant of it from \cite{G73}.

\begin{lemma}[Doob's martingale inequality] \label{lmm: Doob's}
    Let $(Z_n)_{n\geq 0}\in \RR^d$ be a martingale. Then
    $$\Prob\left(\max_{0\leq k \leq n} \| Z_k\|\geq \lambda\right)\leq \frac{C_r}{\lambda^2}\EE\left[\|Z_n\|^2\right]$$
    for every $\lambda >0$, where $C_d$ is a constant that depends on the dimension.
\end{lemma}

For our purposes, we approximate $B_1+\cdots + B_n$ by the martingale
$$Z_n := \sum_{i=1}^n (B_i - \EE[B_i \mid B_{i-1}, \dots, B_1]) = B_1+\cdots + B_n + R_n.$$
Using \cref{prop: independence} we can estimate each error term by
\begin{align*}
    \|R_n\|\leq \sum_{i=1}^n \left\|\EE[B_i \mid B_{i-1}, \dots, B_1]\right\| \leq \sum_{i=1}^n c^{a_i} \leq \sum_{i=1}^\infty c^{i^{1/3}} <\infty.
\end{align*}
This and \eqref{eq: total mixing error} give $\|X_{s_k}-M_k\| \leq \|B_1 + \cdots + B_k\| + O(1) \leq \|Z_k\| + O(1)$. Thus, for $\eps > 0$
\begin{equation}
    \begin{split}\label{eq: union bound}
        \Prob\left(\frac{1}{\sqrt{s_n}} \max_{1\leq k\leq n}\|X_{s_k} - M_k\|\geq \eps\right) &\leq \Prob\left(\max_{1\leq k\leq n}\left(\|Z_k\| + O(1)\right) \geq \eps\sqrt{s_n}\right)\\
        &\leq \Prob\left(\max_{1\leq k\leq n} \|Z_k\|\geq \frac{\eps\sqrt{s_n}-O(1)}{2} \right) \\
        &\leq \frac{4C_d}{(\eps \sqrt{s_n}-O(1))^2}\EE\left[\|Z_n\|^2\right],
    \end{split}
\end{equation}
by \cref{lmm: Doob's}. Finally, to upper bound the right-hand side, we expand
\begin{align*}
    \EE\left[\|Z_n\|^2\right] = \sum_{i=1}^n \EE\left[\|B_i\|^2\right] + \sum_{i=1}^n\EE\left[B_i^\top R_n\right] +\sum_{i=1}^n \sum_{\substack{1\leq j\leq n \\ j\neq i}} \EE\left[B_iB_j^\top\right] + \EE\left[\|R_n\|^2\right].
\end{align*}
We estimate each term in this sum separately. 
\begin{itemize}
    \item Since $B_i\leq O(\log i)$, the first sum is at most $n\cdot O(\log^2 n)$.
    \item By Cauchy-Schwarz $|\EE[B_i^\top R_n]|\leq \EE[\|B_i\|]\cdot \EE[\|R_n\|]$, so the second sum is $O(n\log n)$. 
    \item Fix $i$. For each $j$ the events $B_i$ and $B_j$ are at distance at least $|s_i-s_j|$, so
    $$\|\EE[B_iB_j^\top]\| \leq \EE\left[\|B_i\|\cdot \|\EE[B_j \mid B_i]\|\right]\leq Cc^{|s_i-s_j|} O(\log n).$$
    This decays exponentially in $j$, so summing over all $j$ gives a bound of $O(\log n)$. 
    \item We always have $\|R_n\| \leq \Theta(n\log n)$, so $\EE[\|R_n\|^2]\leq \EE[\|R_n\|]\cdot \Theta(n\log n) = O(n\log n)$.
\end{itemize}
Thus $\EE[\|Z_n\|^2] = O(n\log^2 n)$, so the upper bound in~\eqref{eq: union bound} is $O(n^{-4/3})\cdot O(n\log^2 n) = o(1)$.

\section{Further Directions} \label{sec: further directions}

At present, little is known about the variance $\sigma_{b, p}^2$. Defant, Jiradilok and Mossel \cite{DJM25} show that, for certain directions $b$ corresponding to \dfn{Coxeter words}, the variance takes a remarkably simple form: $\sigma_{b, p}^2$ is a scalar multiple of $p/(1-p)$. It would be natural to investigate properties of $\sigma_{b, p}^2$ now that we know it extends continuously to $S^{d-1} \times (0,1)$.

\begin{conj}
    For every fixed $b\in S^{d-1}$, $\sigma_{b, p}^2$ is monotone in $p$. 
\end{conj}

It would also be interesting to find which directions $b$ maximize $\sigma_{b, p}^2$, for each $p$. Another direction is studying its regularity. The discussion in \cref{rmk: quotient space X} leads to our next conjecture.

\begin{conj}\label{conj: discontinuous}
For fixed $p$, the map $b \mapsto \sigma_{b, p}$ is discontinuous at rational directions.
\end{conj}

\begin{remark}
    A natural approach to proving \cref{thm: main theorem} would be to approximate an irrational direction $b$ by a sequence $(b_n)_{n \geq 0}$ of rational directions and apply a coupling argument. If \cref{conj: discontinuous} holds, this strategy is likely to fail.
\end{remark}

As a natural variant of the random reflection billiard walk, we introduce the \dfn{random refraction billiard walk} $(R_t)_{t\geq 0}$, where reflections are replaced by \dfn{refractions} (see~\cref{fig:refractions}). This model follows the notion of refraction with index $-1$, as introduced to algebraic combinatorics in~\cite{ADS24} and further developed in~\cite{DL25}. 

\begin{figure}[H]
    \centering
    \begin{subfigure}[b]{0.35\textwidth}
        \centering
        \includegraphics[width = \linewidth]{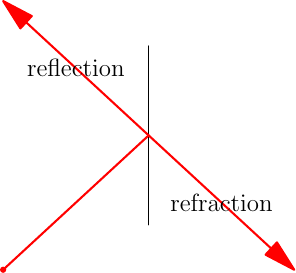}
        \caption{Definition of refractions}
        \label{fig:refractions}
    \end{subfigure}
    \hfill
    \begin{subfigure}[b]{0.58\textwidth}
        \centering
        \includegraphics[width = \linewidth]{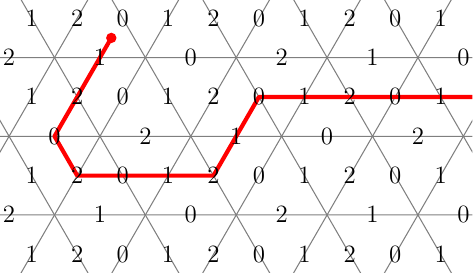}
        \caption{Sample path in $\widetilde{A}_2$.}
        \label{fig: refraction billiard}
    \end{subfigure}

    \caption{Description of the Random Refraction Billiard Walk}
    \label{fig: random refraction billiard walk}
\end{figure}

As shown in \cref{fig: refraction billiard}, each refraction reverses the walk's direction along the cutting sequence, breaking the Markov property of $(w_n)_{n\geq 0}$. Nevertheless, simulations and preliminary results suggest the random refraction billiard walk has similar limiting behavior.

\begin{conj}
    For every $R_0\in \RR^d$, $b\in S^{d-1}$ and $p\in (0, 1)$, the rescaled trajectories
    $$R^{(n)}_t = \frac{1}{\sqrt{n}} R_{tn}, \quad t\in [0, 1]$$
    converge to isotropic Brownian motion in $\RR^d$. 
\end{conj}

\section{Acknowledgements}

This research was conducted at the University of Minnesota Duluth REU, which is supported by Jane Street Capital, NSF Grant 2409861, and donations from Ray Sidney and Eric Wepsic. I am very grateful to Joe Gallian and Colin Defant for providing this wonderful opportunity. I would like to thank Noah Kravitz, Mitchell Lee, Rupert Li and Daniel Zhu for helpful discussions, as well as Eliot Hodges and Carl Schildkraut for feedback on the paper.

\end{document}